\documentclass[11pt]{amsart}
\usepackage{multirow,bigdelim}
\usepackage{amsfonts}
\usepackage{dsfont}
\usepackage{amsmath}
\usepackage{mathrsfs}
\usepackage{todonotes}
\usepackage{hyperref}
\usepackage{multimedia}
\usepackage{graphicx}
\usepackage{indentfirst}
\usepackage{bm}
\usepackage{amsthm}
\usepackage{mathrsfs}
\usepackage{latexsym}
\usepackage{amsmath}
\usepackage{amssymb}
\usepackage{microtype}
\usepackage{relsize}
\usepackage{hyperref}

\DeclareSymbolFont{SY}{U}{psy}{m}{n}
\DeclareMathSymbol{\emptyset}{\mathord}{SY}{'306}

\theoremstyle{plain}

\newtheorem{thm}{Theorem}[section]
\newtheorem{cor}[thm]{Corollary}
\newtheorem{lem}[thm]{Lemma}
\newtheorem{prop}[thm]{Proposition}
\newtheorem{defn}[thm]{Definition}

\theoremstyle{definition}
\newtheorem{rem}[thm]{Remark}
\newtheorem{ex}[thm]{Example}

\numberwithin{equation}{section}

\def\beq{\begin{eqnarray}}
\def\eeq{\end{eqnarray}}
\def\beqa{\begin{eqnarray*}}
\def\eeqa{\end{eqnarray*}}

\begin{document}
\title[On the similarity of operators restricted to an invariant subspace]{On the similarity of operators restricted to an invariant subspace}

\author{Kui Ji, Shanshan Ji$^{*}$, Dinesh Kumar Keshari and Jing Xu}
\curraddr[K. Ji, S. Ji and J. Xu ]{Department of Mathematics, Hebei Normal University,
Shijiazhuang, Hebei 050016, China}

\email[K. Ji]{jikui@hebtu.edu.cn, jikuikui@163.com}
\email[S. Ji]{jishanshan15@outlook.com}
\email[D.K. Keshari]{dinesh@niser.ac.in}
\email[J. Xu]{xujingmath@outlook.com}
\thanks{The first author was supported by National Natural Science Foundation of China (Grant No. 11831006 and 11922108).}



\begin{abstract}
Let $M_{z}$ be the multiplication operator on the Bergman space and $M_{I}$ denote the restriction of $M_{z}$ to an invariant subspace $I$. A question raised by K. Zhu in \cite{ZKH2} is that when are two restriction operators $M_{I}$ and $M_{J}$ are similar? In this note, we give some sufficient conditions of this problem in a general case.
\end{abstract}

\maketitle

\section{Introduction}

Let $\mathbb{D}$ be the open unit disk in the complex plane $\mathbb{C}$,
$\mathcal{L}(\mathcal{H})$ be the set of all bounded linear operators on the infinite dimensional complex separable Hilbert space $\mathcal{H}$.
One of the basic problems of operator theory is how to judge the similarity of two bounded linear operators. Let $T$ and $S$
 belong to $\mathcal{L}(\mathcal{H})$ and $\mathcal{L}(\mathcal{K})$, respectively. We say $T$ is similar to $S$ (denoted by $T\sim_{s}S$), if there
exists an invertible bounded linear operator $X: \mathcal{H}\rightarrow \mathcal{K}$, such that $XT=SX$.
We say $T$ is unitarily equivalent to $S$ (denoted by $T\sim_{u}S$), if $X$ is a unitary operator.
Compared with the unitary equivalence of operators, the similarity of operators is more difficult.
In general, one can only study more specific subclasses.

The Beurling theorem, due to A. Beurling, stated that invariant subspaces of multiplicative operator on Hardy space can be characterized by inner functions (cf.\cite{BA}).
According to Beurling's theorem,
P. Bourdon proved in \cite{Bourdon} that the multiplication operator on Hardy space is unitarily equivalent to its restrictions on each invariant subspace, but there is no such conclusion in Bergman space.
Hardy space and Bergman space are essentially different.
Let $M_{I}$ denote the restriction of the multiplication operator $M_{z}$ to an invariant subspace $I$. For Bergman space, in \cite{RS}, S. Richter gave the following result for the general situation:
\begin{thm}\cite{RS}
Let $M_{z}$ be the multiplication operator on Bergman space. Then $M_{I}$ and $M_{J}$ are unitarily equivalent if and only if $I=J$.
\end{thm}
This theorem tells us that the purpose of studying the properties of restriction operator $M_{I}$ is to better understand the lattice of invariant subspace of multiplication operator on Bergman space.
Let $Lat(T)$ be the lattice of invariant subspace of operator $T$.
In 1997, K. Zhu proposed the following two open problems about the restricted operators $M_{I}$ and $M_{J}$ in \cite{ZKH2}.

\textbf{Problem: }
\begin{enumerate}
  \item[(1)] When are two restriction operators $M_{I}$ and $M_{J}$ similar?
  \item[(2)] When are two restriction operators $M_{I}$ and $M_{J}$ quasi-similar?
\end{enumerate}
Based on these two questions, the next three theorems give some partial answers. For quasi-similar case, H. Hedenmalm, B. Korenblum and K. Zhu proved the following result.
\begin{thm}\cite{HKZ}
Let $M_{z}$ be the multiplication operator on Bergman space. If $I\in Lat(M_{z})$, then $M_{I}$ and $M_{z}$ are quasi-similar if and only if $I$ is generated by a
bounded analytic function.
\end{thm}
Similarly, when $J$ is the whole space, P. Bourdon solved the open problems of similar case using inner-outer factorization in \cite{Bourdon}.
\begin{thm}\cite{Bourdon}\label{2023.02211}
Let $M_{z}$ be the multiplication operator on Bergman space. If $I\in Lat(M_{z})$, then $M_{I}$ and $M_{z}$ are similar if and only if $I$ is generated by a Blaschke product whose zero set is the union of finitely many interpolating sequences.
\end{thm}
Thus, Theorem \ref{2023.02211} shows that the inner functions have played an important role in two problems due to K.Zhu in \cite{ZKH2}.
It is well known that $\eta_{\alpha}(z)=\exp(-\alpha\frac{1+z}{1-z})$ is typical singular inner functions for some $\alpha>0$.
B. Korenblum found two invariant subspaces $I_{1}, I_{2}$ by such functions and gave a sufficient condition for Problem (2).
\begin{thm}\cite{BK}
For any two positive numbers $\alpha_{1}$ and $\alpha_{2}$, set $S_{k}(z)=\exp(-\alpha_{k}\frac{1+z}{1-z}), k=1, 2$.
If $I_{1}$ and $I_{2}$ be the invariant subspaces generated by $S_{1}$ and $S_{2}$, respectively, then $M_{I_{1}}$ and $M_{I_{2}}$ are similar.
\end{thm}
These three theorems partly answer two open problems proposed by K. Zhu from the perspective of function theory. However, these two problems are still unsolved.
For the multiplication operator $M_{z}$ on some analytic functional Hilbert space,
if an operator $T$ is similar to $M_{z}^{*}\oplus M_{z}^{*}$, it is not clear for an invariant subspace $\mathcal{M}$ of $T$, when the restriction $T|_{\mathcal{M}}$ is similar to $M_{z}^{*}$.

In 1978, M. J. Cowen and R. G. Douglas introduced a class of operators denoted by $B_n(\Omega)$ in \cite{CD}.
As we all know that the adjoint of multiplication operator on the Bergman space belongs to $B_1(\mathbb{D})$. For the similarity of operators in $B_n(\mathbb{D})$ and even $B_1(\mathbb{D})$ is still an open problem raised by M. J. Cowen and R. G. Douglas.

Recently,  R. G. Douglas, H. Kwon, S. Treil, Y. Hou and the first author in \cite{DKS, HJK} described the similar invariants of $m-$hypercontractive operators in $B_n(\mathbb{D})$ and weighted Bergman shifts by curvatures of the Hermitian holomorphic bundles induced by corresponding operators.

\begin{thm}\cite{DKS, HJK}\label{DKS}
Let $T \in B_n(\mathbb{D})$ be a $m$-hypercontraction and $M_z$ is the multiplication operator on the weighted Bergman space.
Then $T \sim_{s}\bigoplus\limits^n_{i=1}M^*_{z}$ if and only if there exists a bounded subharmonic function
$\psi$ such that $$\mbox{trace}\,\mathcal{K}_{M^{*(n)}_z}(\omega)-\mbox{trace}\,\mathcal{K}_T(\omega)=\frac{\partial ^{2}}{\partial \omega\partial\bar{\omega}}\psi(\omega), \omega\in \mathbb{D}.$$
\end{thm}
The above theorem gives a characterization of operator similarity in the language of complex geometry.
In this note, we will consider the general case of K. Zhu's questions using geometric quantity. When two restrictions of operators which belongs to $B_n(\mathbb{D})$ are similar?
In the main theorem, we will give some sufficient conditions in the terms of quotient of metrics of some jet Hermitian holomorphic bundles.
The paper is organized as follows. In Section 2, we introduce some notations and definitions about Cowen-Douglas operators.
In Section 3, influenced by the open problems proposed by K. Zhu, we investigate the similarity between the Cowen-Douglas operators restricted on  invariant subspaces and the direct sum of multiplication operator. At the same time, we give a sufficient condition for open problems in the case of similar equivalence.
In Section 4, we consider the reducibility of the $n$-hypercontractive operators by investigating the curvatures of the operators restricted on certain invariant subspaces.
\section{Preliminaries}
\subsection{Cowen-Douglas operator.}
An operator $T\in \mathcal{L}(\mathcal H)$ is said to be a Cowen-Douglas operator with index $n$ associated with a bounded connected open subset $\Omega$ of $\mathbb{C}$ (or $T\in B_n(\Omega)$), if $T-w$ is surjective, $\dim \mbox{ker}(T-\omega)=n$ for all $\omega\in \Omega$ and $\bigvee\limits_{\omega\in \Omega} \mbox{ker}(T-\omega)=\mathcal{H}$.
In \cite{CD}, M. J. Cowen and R. G. Douglas showed that each operator $T$ in $B_n(\Omega)$ also yields a rank $n$ Hermitian holomorphic vector bundle $E_T$ over $\Omega$. That is,
$$E_{T}=\{(x,w)\in {\mathcal H}\times \Omega \mid x\in \ker (T-w)\}
\,\,\mbox{and} \,\,\pi:E_T\rightarrow \Omega,\,\,\mbox{where}
\,\,\pi(x,w)=w.$$
Let $T\in B_n(\Omega)$. Due to $\dim \mbox{ker}(T-\omega)=n$,
we can find a holomorphic family $\{\gamma_{i}\}_{i=1}^{n}$ satisfying $\gamma_{i}(\omega)\in\mbox{ker}(T-\omega), \omega\in\Omega$.
The metric of holomorphic bundle $E_T$ at $\omega$ is a Gram matrix in the form of
$$h_T(\omega)=\big (\!\!\big (\langle \gamma_j(\omega),\gamma_i(\omega)\rangle \big ) \!\!\big )_{i,j=1}^n.$$
Then the curvature function $\mathcal{K}_T$ of $E_T$ is defined as
$\mathcal{K}_T(\omega):=-\frac{\partial}{\partial \overline{\omega}}\big(h_T^{-1}(\omega)\frac{\partial}{\partial \omega}h_T(\omega)\big)$.
Notice that the curvature can be thought of as a bundle map. According to the definition of covariant partial derivative of bundle map,
the covariant partial derivatives $ \mathcal{K}_{T, \omega^i\overline{\omega}^j}$, $i,j\in
\mathbb{N}\cup \{0\}$ of $\mathcal{K}_T$ are defined as follows:

 (1)\,\,$\mathcal{K}_{T,\omega^i\overline{\omega}^{j+1}}=\frac{\partial}{\partial \overline{\omega}}(\mathcal{K}_{T,\omega^i\overline{\omega}^{j}});$

 (2)\,\,$\mathcal{K}_{T,\omega^{i+1}\overline{\omega}^j}=\frac{\partial}{\partial \omega}(\mathcal{K}_{T,\omega^{i}\overline{\omega}^j})+[h_{T}^{-1}\frac{\partial}{\partial \omega}h_{T},\mathcal{K}_{T,\omega^{i}\overline{\omega}^j}].$

M. J. Cowen and R. G. Douglas gave the following important theorem in \cite{CD}, and showed that the curvature and its covariant partial derivatives are complete  unitary invariants of operators in the Cowen-Douglas class.
\begin{thm}\cite{CD}\label{CDT1}
Let $T$ and $\tilde{T}$ be two Cowen-Douglas operators with index n. Then $T\sim_{u}\tilde{T}$ if
and only if there exists an isometric bundle map $V: E_{T}\rightarrow E_{\tilde{T}}$
such that
$$V \big ( {\mathcal K}_{T,w^i\overline{w}^j} \big )=\big ( {\mathcal K}_{\tilde{T},w^i\overline{w}^j}\big ) V,\;
\, 0\leq i, j\leq i+j\leq n, \, (i,j)\neq (0,n),(n,0).$$
\end{thm}

In particular, if $T$ and $\tilde{T}$ are Cowen-Douglas operators with index one, then $T\sim_u \tilde{T}$ if and only if $\mathcal{K}_{T}=\mathcal{K}_{\tilde{T}}$.
It is shown in Theorem \ref{CDT1} that the local geometric invariants of Cowen-Douglas operators can be regarded as global invariants from the perspective of unitary equivalence.
Since invertible holomorphic bundle map can destroy rigidity.
For similar equivalence, it is not clear how to describe global invariants in terms of local complex invariants of curvature and its covariant partial derivatives.
Subsequently, M. J. Cowen and R. G. Douglas gave a necessary condition of similarity by curvature in the next theorem.
\begin{thm}\cite{CD}\label{CDT2}
Let $T$ and $\tilde{T}$ belong to $B_{n}(\Omega)$, which are similar via the bounded invertible operator $X$. Then
$$\|X^{-1}\|^{-2}\|X\|^{-2}\|\mathcal{K}_{T}\|\leq\|\mathcal{K}_{\widetilde{T}}\|\leq\|X^{-1}\|^{2}\|X\|^{2}\|\mathcal{K}_{T}\|,$$
at each point of\, $\Omega$.
\end{thm}
In fact, this necessary condition is not strong enough. Let $T$ and $\tilde{T}$ denote Hardy and Bergman shifts, respectively. We know that
$\frac{\mathcal{K}_{T}}{\mathcal{K}_{\widetilde{T}}}=\frac{1}{2}$. But they are not similar.
Based on this, M. J. Cowen and R. G. Douglas made following conjecture in \cite{CD}.
That is, let $T_{i} \in B_1(\mathbb{D})$ with the spectrums of $T_{i}$ are closure of unit disk $\mathbb{D}, i=1, 2$, then $T_{1}\sim_s T_{2}$ if and only if $\lim\limits_{\omega\rightarrow \partial \mathbb{D}}\frac{\mathcal{K}_{T_{1}}(\omega)}{\mathcal{K}_{T_{2}}(\omega)}=1.$
This conjecture was shown by D. N. Clark and G. Misra to be inaccurate in \cite{CM0, CM}.

In \cite{ZKH}, K. Zhu introduced a criterion to determine when two Cowen Douglas operators are similar (unitary equivalent).
Let $T\in B(\Omega)$. A holomorphic function $\gamma:\Omega\rightarrow \mathcal{H}$, $\gamma(w)$ in the fiber $E_T(w)$, is said to be a spanning holomorphic cross-section for $E_{T}$,
if $\overline{span}\{\gamma(w):w\in\Omega\}=\mathcal{H}$.
The main results of \cite{ZKH} proved that every Hermitian holomorphic vector bundle corresponding to the Cowen-Douglas operator has a spanning holomorphic cross-section and the similarity or unitary equivalent of the Cowen-Douglas operator can be described by it.

We know that Cowen-Douglas operators with index one contain a large number of weighted backward shift operators.
The following well-known theorem, given by A. L. Shields, describes the equivalent condition for the similarity of two weighted shift operators.
\begin{thm}\label{alger1}\cite{SAL}
Let $S$ and $T$ be weighted shift operators with non-zero weight sequences $\{\upsilon_{n}\}_{n=0}^{\infty}$ and $\{\omega_{n}\}_{n=0}^{\infty}$, respectively. Then $S\sim_{s}T$ if and only if there exist real numbers $m$ and $M$, such that
$$0<m\leq\left|\frac{\omega_{i}\omega_{i+1}\cdots\omega_{j}}{\upsilon_{i}\upsilon_{i+1}\cdots\upsilon_{j}}\right|\leq M,\ 0\leq i\leq j.$$
\end{thm}

In reality, for most operators, its structure is quite complex.
To make more progress in similar case, we need the following theorem to simplify the structure of Cowen-Douglas operators.
\begin{thm}[{Upper triangular representation theorem}, \cite{JW}]\label{(jiang1)}
Let $T\in\mathcal{L}(\mathcal{H})$ be a Cowen-Douglas operator with index $n$,
then there exists a orthogonal decomposition $\mathcal{H}=\mathcal{H}_1\oplus\mathcal{H}_2\oplus\cdots\oplus\mathcal{H}_n$ and  operators $T_{1,1}, T_{2,2}, \cdots, T_{n,n}$ in $B_1(\Omega)$ such that $T$ takes following form

\begin{equation} \label{1.1T}
T=\left ( \begin{smallmatrix}T_{1,1} & T_{1,2}& T_{1,3}& \cdots & T_{1,n}\\
0&T_{2,2}&T_{2,3}&\cdots&T_{2,n}\\
\vdots&\ddots&\ddots&\ddots&\vdots\\
0&\cdots&0&T_{n-1,n-1}&T_{n-1,n}\\
0&\cdots&\cdots&0&T_{n,n}
\end{smallmatrix}\right ).
\end{equation}
\end{thm}
\subsection{Reproducing kernel and operator Model theorem.}\label{2.2}
Let $M_n(\mathbb{C})$ denote the set of all $n\times n$ matrices over $\mathbb{C}$, let $\Omega$ be a bounded domain of $\mathbb{C}$.
A function $K:\Omega\times\Omega\rightarrow M_n(\mathbb{C})$ is said to be a non-negative definite kernel if for any subset $\{w_1,\cdots,w_k\}\subset\Omega$,
the $k\times k$ block matrix $(K(w_i,w_j))_{1\leq i,j\leq k}$ is non-negative definite, that is, $\sum_{i,j=1}^k\langle K(w_i,w_j)\eta_j,\eta_i\rangle\geq0,\eta_1,\cdots,\eta_k\in\mathbb{C}^n$.
Similarly, $K$ is said to be a positive definite kernel if the block matrix above is non-negative definite and invertible. In this paper, we always assume that the kernel function $K$ $K$ is holomorphic in the first variable and anti-holomorphic in the second variable, in short, it is sesqui-analytic.

Given a Hilbert space containing $\mathbb{C}^{n}$-valued analytic functions over $\Omega$. Let $e_w$ be the evaluation map from $\mathcal{H}$ to $\mathbb{C}^{n}$,
defined by $e_w(f)=f(w),~ f\in\mathcal{H}$, is bounded for any $w\in\Omega$.
For vector $\xi\in\mathbb{C}^{n}$, the function $e^*_w\xi\in\mathcal{H}$ and is denoted by $K(\cdot,w)\xi$, which has the reproducing property
$\langle f,K(\cdot,w)\xi\rangle_\mathcal{H}=\langle f(w),\xi\rangle_{\mathbb{C}^{n}}$ for any $f\in\mathcal{H}$.
For fixed but arbitrary $z,w\in\Omega$, $K(z,w)=e_ze^*_w$ is an $n\times n$ matrix.
Then the matrix-valued kernel function $K$ is called the reproducing kernel of $\mathcal{H}$ and $\mathcal{H}$ is called the reproducing kernel Hilbert space.
It is easy to verify that every reproducing kernel is a non-negative definite kernel.

There is a close relationship between the operators in the Cowen-Douglas class and adjoints of the multiplication operators on the reproducing kernel Hilbert spaces.
Let $T\in B_n(\Omega)$. Then $T$ is unitarily equivalent to
the adjoint of multiplication operator $M_z$ on a Hilbert
space $\mathcal{H}_K$ of holomorphic functions on $\Omega^*=\{w\in \mathbb{C}^m:\overline{w}\in \Omega\}$ with reproducing kernel $K$ in \cite{CD,CS2,ZKH2} and
$\ker(M_z^*-w)=\{K(\cdot, \overline{w})\xi, \xi\in \mathbb{C}^n\}$.
Conversely, with mild assumptions on the non-negative kernel $K:\Omega\times\Omega\rightarrow M_n(\mathbb{C})$, one may assume
the adjoint of multiplication operator $M_z$ acting on Hilbert space $\mathcal{H}_K$ determined by reproducing kernel $K$ belongs to $B_n(\Omega)$ (see \cite{CS2}).
Throughout this article, we require that these assumptions remain valid.

In order to generalize the model theorem given by B. Sz.-Nagy and C. Foias, J. Agler introduced the concept of $n-$hypercontraction in \cite{Agler1}.
We say $T\in\mathcal{L}(\mathcal{H})$ is an $n-$hypercontractive operator, if
$$\sum\limits_{j=0}^{k}(-1)^{j}{k \choose j}(T^{*})^{j}T^{j}$$ is positive for any $1\leq k\leq n$. When $n=1$, $T$ is called a contraction.
Corresponding to the $n-$hypercontraction, the following model theorem is widely used. That's the basis of the similarity theorem \ref{DKS}.
\begin{thm}[{Model Theorem}, \cite{Agler2}]\label{9.211}
Let $T\in\mathcal{L}(\mathcal{H})$ and $M_{z}$ is the multiplication operator on the Hilbert space $\mathcal{H}_{K}$ with reproducing kernel $K(z, \omega)=\frac{1}{(1-z\overline{\omega})^{n}}$, $z,w\in\mathbb{D}$.
Then there exists a Hilbert space $E$ and an $M_{z, E}^{*}-$invariant subspace $\mathcal{N}\subset\mathcal{H}_{K}\otimes E$, such that $T$ is unitarily equivalent to $M_{z, E}^{*}|_{\mathcal{N}}$ if and only if $T$ is an $n-$hypercontraction with $\lim\limits_{n}\|T^{n}h\|=0$ for every $h\in \mathcal{H}$.
\end{thm}
The symbol $M_{z, E}$ is represented as a multiplication operator on the vector valued space $\mathcal{H}_{K}\otimes E$.
When $E=\mathbb{C}^n$, $M_{z, E}$ can also be written as $\bigoplus\limits^n M_z$.
Then the eigenvector bundle of $T$ has the tensor structure
$\ker(T-\omega I)=span\{K(\cdot,\overline{\omega})\}\otimes \mathcal{E}(\omega)$,
where $\mathcal{E}(\omega)=\{e\in E: K(\cdot,\overline{\omega})e\in\mathcal{N}\}$.
If $t\in E_{T}$, by Remark 4.3 of \cite{JKSX1}, we have
$$\|t(\omega)\|^{2}=\frac{1}{(1-|\omega|^{2})^{n}}\biggl\|\biggl(\sum\limits_{j=0}^{n}(-1)^{j}{n \choose j}(T^{*})^{j}T^{j}\biggl)^{\frac{1}{2}}t(\omega)\biggl\|^{2}.$$
C. Ambrozie, M. Engli$\check{s}$ and V. M$\ddot{u}$ller had extended the above kind of model to the most general case, where the model has become multiplication operators on the Hilbert space of analytic functions whose reciprocal of the reproducing kernel is polynomial in \cite{AEM}.
In the following, we will introduce another version of the model theorem given by J. Arazy and M. Engli$\check{s}$ in \cite{AE}, which is more applicable.
This model theorem covers the case where the kernel function is $K(z,w)=\frac{1}{(1-z\overline{w})^\alpha},z,w\in\mathbb{D}$, $\alpha$ is not an integer and $\frac{1}{K}$ is not a polynomial.
Before we do that, let's recall some notations.

Given a Hilbert space $\mathcal{H}_K$ determined by reproducing kernel $K$ over $\Omega$, which satisfies $M_z\mathcal{H}_K\subset\mathcal{H}_K$ and all polynomials are dense in $\mathcal{H}_K$.
In light of the Gram-Schmidt orthogonalization process, we may find an orthonormal basis $\{q_l(z),l\geq0\}$ of $\mathcal{H}_{K}$ such that $K(z,w)=\sum_{l=0}^\infty q_l(z)\overline{q_l(w)}$.
Assume that $K$ is invertible and that there exists a sequence of $\{p_l(z,w),l\geq0\}$, holomorphic in $z$ and $\overline{w}$, such that
$p_l(z,\overline{w})\rightarrow\frac{1}{K}(z,w)$ as $l\rightarrow\infty$, $z,w\in\Omega$.
For any nonnegative operator $C$ and $T\in\mathcal{L}(\mathcal{H})$, set $f_{m,C}(T)=I-\sum\limits_{0\leq j<m}q_j(T)^*Cq_j(T)$.
The following theorem is based on these notations.

\begin{thm}[{Model Theorem}, \cite{AE}]\label{1.282}
Let $T\in\mathcal{L}(\mathcal{H})$ and $M_z$ be the multiplication operator on $\mathcal{H}_{K}$. Then there exist a Hilbert space $E$ and an $M_{z, E}^{*}-$invariant subspace $\mathcal{N}\subset\mathcal{H}_{K}\otimes E$ such that $T\sim_{u}M_{z, E}^{*}|_{\mathcal{N}}$
if and only if $\sup\limits\|p_l(T^*,T)\|<\infty$, $WOT-\lim\limits_{l\rightarrow\infty}p_l(T^*,T):=C$ exist and is nonnegative, and $f_{m,C}(T)\rightarrow0$ strongly as $m\rightarrow\infty$,
where WOT stands for weak operator topology.
\end{thm}
It is proved in Proposition 5.2 of \cite{JiSar} due to the first author and J. Sarkar that $f_{m,C}(T)\rightarrow0$ strongly as $m\rightarrow\infty$ can be removed when $T\in B_{n}(\mathbb{D})$ satisfies $\sup\limits\|p_l(T^*,T)\|<\infty$ and there exist a nonnegative operator $C$ such that $WOT-\lim\limits_{l\rightarrow\infty}p_l(T^*,T):=C$.

\subsection{Uniqueness of strongly irreducible decomposition up to similarity}
Let $\mathcal{A}'(T)=\{X\in \mathcal{L}(\mathcal{H})|TX=XT\}$ be the commutant of $T\in \mathcal{L}(\mathcal{H})$.
We say $T\in{\mathcal L}({\mathcal H})$ is strongly irreducible, if $\mathcal{A}'(T)$ doesn't have any nontrivial idempotent operator (see \cite{G5,Halmos}).
In fact, strongly irreducible operators are Jordan blocks and ``minimal" operators in infinite dimensional space.
In what follows, $T\in(SI)$ means that $T$ is a strongly irreducible operator.
C. K. Fong and C. L. Jiang in \cite{FJ} showed that Cowen-Douglas operators with index one are typical strongly irreducible operators.
In \cite{CFJ}, Y. Cao, J. S. Fang and C. L. Jiang introduced the $K-$theory for the first time in the study of commutant algebra of operators.
The following are some notations  and related results about K-theory.
Let $M_{k}(\mathcal{A}'(T))$ be the collection of all $k\times k$ matrices with entries from the $\mathcal{A}'(T)$ and
$$M_{\infty}(\mathcal{A}'(T))=\bigcup\limits_{k=1}^{\infty} M_{k}(\mathcal{A}'(T)). $$
Let $\text{Proj}(M_{k}(\mathcal{A}'(T)))$ be the algebraic equivalence classes of idempotents in $M_{k}(\mathcal{A}'(T))$ and
$$\bigvee(\mathcal{A}'(T))=\text{Proj}(M_{\infty}(\mathcal{A}'(T))).$$
The $K_{0}-$group of $\mathcal{A}'(T)$ (denoted by $K_0({\mathcal A}^{\prime}(T))$) is defined as the Grothendieck group of $\bigvee ({\mathcal A}^{\prime}(T))$.

\begin{defn}\cite{CFJ}\label{12.11}
Let $T\in \mathcal{L}(\mathcal{H})$. We say $\mathcal{P}=\{P_{i}\}_{i=1}^{n}(n<\infty)$, the set of idempotent elements of $\mathcal{L}(\mathcal{H})$, is a unit finite decomposition of $T$, if the following are satisfied:
\begin{itemize}
  \item [(1)] $P_i{\in} {\mathcal A}^{\prime}(T), 1{\leq}i{\leq}n$;
  \item [(2)] $P_iP_j={\delta}_{ij}P_i$\,\,~$(1{\leq}i, j{\leq}n)$, where
${\delta}_{ij}=\left\{
\begin{array}{cc}
1 &\ i=j\\
0 &\ i \neq j
\end{array}
\right. $;
  \item [(3)] $\sum\limits_{i=1}^{n}P_i=I_{\mathcal H},$\,\,where $I_{\mathcal H}$ denotes identity operator on $\mathcal H$.
\end{itemize}
In addition, the following is satisfied:
\begin{itemize}
  \item [(4)] $T|_{ran P_i}\in(SI), 1 \leq i \leq n$,
\end{itemize}
we call ${\mathcal P}=\{P_i\}_{i=1}^{n}$ is a unit finite (SI) decomposition of $T$.
\end{defn}
By Definition \ref{12.11}, we know that $T\in \mathcal{L}(\mathcal{H})$ has a unit finite (SI) decomposition is equivalent to $T$ can be written as a direct sum of finite strongly irreducible operators.
Naturally, for any $T\in B_{n}(\Omega)$, $T$ has a finite (SI) decomposition. 
\begin{defn}\cite{CFJ}\label{UD}
Let $T$ have finite (SI) decomposition. If ${\mathcal P}=\{P_i\}_{i=1}^{m}$ and ${\mathcal Q}=\{Q_i\}_{i=1}^{n}$ are two unit finite (SI) decompositions of $T$, and the following are satisfied:
\begin{itemize}
  \item [(1)] $m=n$;
  \item [(2)] there exists an invertible operator $X \in {\mathcal A}^{\prime}(T)$ and a permutation ${\Pi}$ in $\{1,2,\cdots,n\}$ such that $XQ_{{\Pi}(i)}X^{-1}=P_i\,\,(1{\leq}i{\leq}n),$
\end{itemize}
then $T$ has unique finite (SI) decomposition up to similarity.
\end{defn}
\begin{thm}\cite{CFJ}\label{9.52}
Let $T\in \mathcal{L}(\mathcal{H})$, and let $\mathcal{H}^{(n)}$ denote the direct sum of $n$ copies of $\mathcal{H}$ and $A^{(n)}$ the operator $\bigoplus\limits_{1}^{n}A$ acting on $\mathcal{H}^{(n)}$. Then the following are equivalent:
\begin{itemize}
  \item[(1)] $T\sim_{s}\bigoplus\limits_{i=1}^{k}A_{i}^{(n_{i})}$ with respect to the decomposition $\bigoplus\limits_{i=1}^{k}\mathcal{H}^{(n_{i})}$, where $k, n_{i}<\infty, A_{i}\in(SI), A_{i}\sim_{s}A_{j}(i\neq j)$, and for each natural number $n$, $T^{(n)}$ has unique finite (SI) decomposition up to similarity;
  \item[(2)] $\bigvee(\mathcal{A}'(T))\cong \mathbb{N}^{k}, K_0({\mathcal A}^{\prime}(T))\cong \mathbb{Z}^{k}$ and this isomorphism $h$ sends $[I]$ to $(n_{1}, n_{2}, \cdots, n_{k})$, i.e., $h([I])=n_{1}e_{1}+n_{2}e_{2}+\cdots+n_{k}e_{k}$, where $\{e_{i}\}_{i=1}^{k}$ are the generators of $\mathbb{N}^{k}$.
\end{itemize}
\end{thm}
Theorem \ref{9.52} transforms the study of operator similarity classification into the isomorphism of the $K_{0}-$group of its commutant algebra.
It is proved that if a bounded operator has a (SI) decomposition, then it is unique up to similarity.
Using Theorem \ref{9.52}, C. L. Jiang obtained the following:
\begin{thm}\cite{Jiang}\label{12.13}
If $T\in B_{n}(\Omega)\cap(SI)$, then $\bigvee(\mathcal{A}'(T))\cong \mathbb{N}, K_{0}(\mathcal{A}'(T))\cong \mathbb{Z}$.
\end{thm}

In \cite{JGJ}, C. L. Jiang, X. Z. Guo and K. Ji characterized a completely similar invariant of Cowen-Douglas operators by the following theorem.
\begin{thm}\cite{JGJ}\label{9.14}
Let $T,~S\in B_{n}(\Omega)$. Suppose that $T=T_{1}^{(n_{1})}\oplus T_{2}^{(n_{2})}\oplus\cdots\oplus T_{k}^{(n_{k})},$ where $0\neq n_{i}\in \mathbb{N}, i=1, 2, \cdots, k$, $T_{i}\in(SI)$, $T_{i}\nsim_{s}T_{j}(i\neq j)$. Then $T\sim_{s}S$ if and only if
\begin{enumerate}
\item [(1)]$(K_{0}~(\mathcal{A}'(T\oplus S)),~V(\mathcal{A}'(T\oplus S)),~ I)\cong(\mathbb{Z}^{k},~\mathbb{N}^{k},~1);$
\item [(2)]The isomorphism $h$ from $\bigvee(\mathcal{A}^{'}(T\oplus S))$ to $\mathbb{N}^{k}$ sends $[I]$ to $(2n_{1}, 2n_{2}, \ldots, 2n_{k}),$ i.e.,
    $$h([I])=2n_{1}e_{1}+2n_{2}e_{2}+\cdots+2n_{k}e_{k},$$ where $I$ is the unit of $\mathcal{A}'(T\oplus S)$ and $\{e_{i}\}_{i=1}^{k}$ are the generators of $\mathbb{N}^{k}$.
\end{enumerate}
\end{thm}

\section{Similarity of operator restriction on an invariant subspace.}
Let $T\in \mathcal{L}(\mathcal{H})$. The closed subspace $\mathcal{M}\subset \mathcal{H}$ is invariant under the operator $T$ if $T\mathcal{M}\subset \mathcal{M}$.
In this section, we will investigate the similarity of operators restricted on invariant subspaces by using operator Model Theorem \ref{1.282}.

We say $T\in \mathcal{L}(\mathcal{H})$ to be Fredholm, if $\dim \mbox{ker}T<\infty, \dim \mbox{coker}T<\infty$, where $\mbox{coker}T=(\mbox{ran}T)^{\bot}$ (cf.\cite{Fred}).
The Fredholm index ind$(T)$ of $T$ is defined as:
ind$(T)=\dim \mbox{ker}T-\dim \mbox{coker}T$.
For $T\in B_{n}(\Omega)$, we have ind$(T-\omega)=\dim \mbox{ker}(T-\omega)=n$, $w\in\Omega$.
\begin{lem}\cite{Fred}\label{yin0}
Let $A, B, C\in \mathcal{L}(\mathcal{H})$ and define $X: \mathcal{H}^{(2)}\rightarrow \mathcal{H}^{(2)}$ by the matrix $X=\left(\begin{matrix}A & B\\
0 & C\\
\end{matrix}\right)$. If $A$ is Fredholm, then $X$ is Fredholm if and only if $C$ is Fredholm.
\end{lem}

\begin{lem}\label{yin1}
Let $T$ be a bounded linear operator of the form $\left(\begin{matrix}T_{1} & T_{12}\\
0 & T_{2}\\
\end{matrix}\right)$. If $T_1\in B_{m}(\Omega)$, then $T\in B_{m+n}(\Omega)$ if and only if $T_{2}\in B_{n}(\Omega)$.
\end{lem}
\begin{proof}
Suppose that $T_1$ and $T_2$ are defined on the Hilbert space $\mathcal{H}_1$ and $\mathcal{H}_2$, respectively. Then $T_{12}\mathcal{H}_2\subset\mathcal{H}_1$ and $T(\mathcal{H}_1\oplus\mathcal{H}_2)\subset(\mathcal{H}_1\oplus\mathcal{H}_2)$.
Since $T\in B_{m+n}(\Omega)$, for any $w\in\Omega$ and $\left(\begin{matrix}y_{1}\\
y_{2} \\
\end{matrix}\right )\in\mathcal{H}_{1}\oplus\mathcal{H}_{2}$, there is $\left(\begin{matrix}x_{1}\\
x_{2} \\
\end{matrix}\right )\in\mathcal{H}_{1}\oplus\mathcal{H}_{2}$ such that
$$\left(\begin{matrix}T_{1}-\omega & T_{12}\\
0 & T_{2}-\omega\\
\end{matrix}\right )\ \left(\begin{matrix}x_{1}\\
x_{2} \\
\end{matrix}\right )=\left(\begin{matrix}(T_{1}-\omega)x_{1}+T_{12}x_{2}\\
(T_{2}-\omega)x_{2} \\
\end{matrix}\right )=\left(\begin{matrix}y_{1}\\
y_{2} \\
\end{matrix}\right ).$$That is to say, for any $w\in\Omega$ and $y_{2}\in \mathcal{H}_{2}$, there is $x_{2}\in \mathcal{H}_{2}$ such that $(T_{2}-\omega)x_{2}=y_{2}$, then
\begin{equation}\label{d1}
\mbox{ran}(T_{2}-\omega)=\mathcal{H}_{2},\ \omega\in\Omega.
\end{equation}
We know that $T-\omega,\ T_{1}-\omega$ are Fredholm, and
$$\mbox{ind}(T-\omega)=\dim \mbox{ker}(T-\omega)=m+n,\ \mbox{ind}(T_{1}-\omega)=\dim \mbox{ker}(T_{1}-\omega)=m.$$
By Lemma \ref{yin0} and equation (\ref{d1}), we have $T_{2}-\omega$ is also a Fredholm operator and ind$(T_{2}-\omega)=\dim \mbox{ker}(T_{2}-\omega)$.
Since $T_{1}\in B_{m}(\Omega)$, there exist a holomorphic frame $\{t_{i}\}_{i=1}^{m}$ of $E_{T_{1}}$. Setting $\gamma_{i}=t_{i}$ for $1\leq i\leq m$.
It is easy to verify that $\gamma_{i}(\omega)\in \ker(T-\omega), 1\leq i\leq m.$ Then we extend the basis $\{\gamma_{i}\}_{i=1}^{m}$ to the holomorphic frame $\{\gamma_{j}\}_{j=1}^{m+n}$ of $E_{T}$, where $\gamma_{j}=\gamma_{j, 1}+\gamma_{j, 2},\ m+1\leq j\leq m+n$. Note that
$$\left(\begin{matrix}T_{1}-\omega & T_{12}\\
0 & T_{2}-\omega\\
\end{matrix}\right )\ \left(\begin{matrix}\gamma_{j, 1}(\omega)\\
\gamma_{j, 2}(\omega) \\
\end{matrix}\right )=\left(\begin{matrix}(T_{1}-\omega)\gamma_{j, 1}(\omega)+T_{12}\gamma_{j, 2}(\omega)\\
(T_{2}-\omega)\gamma_{j, 2}(\omega) \\
\end{matrix}\right )=\left(\begin{matrix}0\\
0 \\
\end{matrix}\right ), $$ that means $\gamma_{j, 2}(\omega)\in \ker(T_{2}-\omega),\ m+1\leq j\leq m+n$.
We claim that $\{\gamma_{j,2}(w), m+1\leq j\leq m+n\}$ is a linear independent set for any $w\in\Omega$.
Without loss of generality, we assume that $\gamma_{m+1,2}(w)$ and $\gamma_{m+2,2}(w)$ are linearly dependence for some $w\in\Omega$,
then vectors $\gamma_{1}(w), \cdots, \gamma_{m+2}(w)$~are also. This is a contradiction.
So $\{\gamma_{j, 2}(\omega),\ m+1\leq j\leq m+n\}$ is a rank $n$ basis of $\ker(T_{2}-\omega)$ for any $w\in\Omega$ and $\Omega\in\sigma(T_2)$.
From $\bigvee \limits_{w{\in}{\Omega}} \ker(T-\omega)=\mathcal{H}_{1}\oplus\mathcal{H}_{2}$, we have
$\bigvee \limits_{w{\in}{\Omega}} \ker(T_2-\omega)=\bigvee \limits_{\omega{\in}{\Omega}} \{\gamma_{j, 2}(\omega)\}_{j=m+1}^{m+n}=\mathcal{H}_{2}$.
Therefore,~$T_{2}\in B_{n}(\Omega)$ and~$\{\gamma_{j, 2}\}_{j=m+1}^{m+n}$~is~the holomorphic frame of $E_{T_{2}}$.

Conversely, suppose that $T_1\in B_{m}(\Omega), T_{2}\in B_{n}(\Omega)$.
Then there exist two holomorphic frames $\{\gamma_{i,1}\}_{i=1}^m$ and $\{\gamma_{i,2}\}_{i=m+1}^{m+n}$ of $E_{T_1}$ and $E_{T_2}$, respectively.
Since $T_1-w$ is surjective for any $w\in\Omega$ and $\{T_{12}\gamma_{i,2}\}_{i=m+1}^{m+n}\subset\mathcal{H}_1$, there exist holomorphic functions $\{\gamma_{i,1}\}_{i=m+1}^{m+n}\subset\mathcal{H}_1$ such that $(T_1-w)\gamma_{i,1}=-T_{12}\gamma_{i,2}$, $m+1\leq i\leq m+n$.
Let $$\gamma_i=
\begin{cases}
\gamma_{i,1},\quad\quad\quad\,\ 1\leq i\leq m,\\
\gamma_{i,1}+\gamma_{i,2},\quad m+1\leq i\leq m+n.
\end{cases}$$
Then $\{\gamma_i(w)\}_{i=1}^{m+n}$ is a set of eigenvectors of $T$ with eigenvalue $w\in\Omega$.
Assume that there exist $\{k_i\}_{i=1}^{m+n}\subset\mathbb{C}$ such that $$\sum_{i=1}^{m+n}k_i\gamma_i=\left(\begin{matrix}\sum\limits_{i=1}^{m+n}k_i\gamma_{i,1}\\
\sum\limits_{i=m+1}^{m+n}k_i\gamma_{i,2} \\
\end{matrix}\right)=\left(\begin{matrix}0\\
0 \\
\end{matrix}\right). $$
It follows from $\{\gamma_{i,1}(w)\}_{i=1}^m$ and $\{\gamma_{i,2}(w)\}_{i=m+1}^{m+n}$ are sets of linearly independent vectors for any $w\in\Omega$, respectively that $k_i=0$, $1\leq i\leq m+n$.
It is proved that the vectors in it are linearly independent of each other and $dim\,\ker(T-w)=m+n$.
Since $\bigvee_{w\in\Omega}\{\gamma_{i,1},1\leq i\leq m\}=\mathcal{H}_1$ and $\bigvee_{w\in\Omega}\{\gamma_{i,2},m+1\leq i\leq m+n\}=\mathcal{H}_2$, we have $\bigvee_{w\in\Omega}\ker(T-w)=\mathcal{H}_1\bigoplus\mathcal{H}_2$.
A routine verification shows from $\mbox{ran}(T_{i}-\omega)=\mathcal{H}_{i},i=1,2$ that $\mbox{ran}(T-\omega)=\mathcal{H}_1\bigoplus\mathcal{H}_2$ for $\omega\in\Omega$.
Hence, $T\in B_{m+n}(\Omega)$.
\end{proof}

For $T$ in Lemma \ref{yin1}, if $T_1$ and $T_2$ are Cowen-Douglas operators, then ind$(T)=\mbox{ind}(T_1)+\mbox{ind}(T_2)$.

\begin{lem}\label{yin2}
Let~$T\in B_{n}(\Omega)\cap \mathcal{L}(\mathcal{H})(n\geq2)$, and $\mathcal{H}$ has an orthogonal decomposition~$\bigoplus\limits_{i=1}^{n}\mathcal{H}_{i}$. Then there exists a holomorphic frame
$\{\gamma_{i}\}_{i=1}^{n}$, such that~$\gamma_{i}=(\gamma_{i, 1}, \cdots, \gamma_{i, i}, \underbrace{0,0,\cdots,0 }_{\text{n-i}})^{T}$, where $\gamma_{i, j}\in \mathcal{H}_{j}, 1\leq i\leq n$.
\end{lem}
\begin{proof}
By Theorem \ref{(jiang1)}, for $T\in B_{n}(\Omega)$, there exist operators $T_{1}, \cdots, T_{n}$ in $\in B_{1}(\Omega)$ such that
$$T=\left(\begin{matrix}T_{1} & T_{1, 2}& \cdots & T_{1, n}\\
0 & T_{2} & \cdots & T_{2, n}\\
\vdots&\vdots & \ddots & \vdots\\0 & 0 & \cdots & T_{n}\\
\end{matrix}\right ) $$with respect to the decomposition $\mathcal{H}=\bigoplus\limits_{i=1}^{n}\mathcal{H}_{i}$, where $T_{i, j}: \mathcal{H}_{j}\rightarrow\mathcal{H}_{i}$.
Let $\gamma_{i, i}\in \ker(T_{i}-\omega)$ be non-vanishing holomorphic sections of bundles $E_{T_{i}}(1\leq i\leq n)$.
In the following, to complete the proof by mathematica induction.

When~$n=2$, we have
$T=\left(\begin{matrix}T_{1} & T_{1, 2}\\
0 & T_{2}\\
\end{matrix}\right )\ \in B_{2}(\Omega)$ and $T_{1}, T_{2}\in B_{1}(\Omega). $
It is easy to verify~$\gamma_{1}(\omega):=\gamma_{1, 1}(\omega)\in \ker(T-\omega)$. Since~$\dim \ker(T-\omega)=2$, suppose that the other non-zero holomorphic vector in the kernel space $\ker(T-\omega)$ is
~$\gamma_{2}=\widetilde{\gamma}_{2, 1}+\widetilde{\gamma}_{2, 2}$. It follows that
$$\left(\begin{matrix}T_{1}-\omega & T_{1, 2}\\
0 & T_{2}-\omega\\
\end{matrix}\right )\ \left(\begin{matrix}\widetilde{\gamma}_{2, 1}(\omega)\\
\widetilde{\gamma}_{2, 2}(\omega) \\
\end{matrix}\right )\ =\left(\begin{matrix}(T_{1}-\omega)\widetilde{\gamma}_{2, 1}(\omega)+T_{1, 2}\widetilde{\gamma}_{2, 2}(\omega)\\
(T_{2}-\omega)\widetilde{\gamma}_{2, 2}(\omega) \\
\end{matrix}\right )\ =\left(\begin{matrix}0\\
0 \\
\end{matrix}\right ),$$and~$\widetilde{\gamma}_{2, 2}(\omega)\in \ker(T_{2}-\omega)$.
We claim that~$\widetilde{\gamma}_{2, 2}\neq0$. Otherwise, $\{\gamma_{1}, \gamma_{2}\}$ is not a frame of~$E_{T}$. Setting~$\gamma_{2, 2}:=\widetilde{\gamma}_{2, 2}$. Then $\gamma_{2, 2}$ is a non-zero cross-section of holomorphic vector bundle $E_{T_{2}}$.
So the conclusion of Lemma holds when $T$ belongs to $B_{2}(\Omega)$.
Let us assume that Lemma is valid for $T\in B_{k-1}(\Omega)$. We will prove that it is also true for $T\in B_{k}(\Omega)$.

Let $T\in B_{k}(\Omega)$. By Theorem \ref{(jiang1)} again, we can write $T$ as $\left(\begin{matrix}T_{1} & \widetilde{T}_{1, 2}\\
0 & \widetilde{T}_{2}\\
\end{matrix}\right)$. According to Lemma \ref{yin1}, we know $\widetilde{T}_{2}\in B_{k-1}(\Omega)$. By the induction hypothesis, there exists a holomorphic frame
$\{\widetilde{\gamma}_{i}\}_{i=2}^{k}$ of bundle $E_{\widetilde{T}_{2}}$, such that $\widetilde{\gamma}_{i}=(\gamma_{i, 2}, \cdots, \gamma_{i, i}, \underbrace{0,0,\cdots,0 }_{\text{k-i}})^{T}$.
Then there exist holomorphic functions $\gamma_{i, 1}$ satisfying $(T_{1}-\omega)\gamma_{i, 1}+\widetilde{T}_{1, 2}\widetilde{\gamma}_{i}=0(2\leq i\leq k)$.
Setting $\gamma_{1}:=\left(\begin{matrix}\gamma_{1, 1}\\
0\\
\end{matrix}\right )$, $\gamma_{i}:=\left(\begin{matrix}\gamma_{i, 1}\\
\widetilde{\gamma}_{i}\\
\end{matrix}\right )(2\leq i\leq k)$. Cleraly, $\{\gamma_{i}\}_{i=1}^{k}$ is a holomorphic frame of $E_{T}$.
This completes the proof.
\end{proof}
\begin{lem}\label{9.2}
Suppose that the Hermitian matrix $A$ can be divided into
$\left(\begin{matrix}A_{11} & A_{12}\\
A_{12}^{*} & A_{22}\\
\end{matrix}\right )$. If $A_{11}$ is invertible, then $A$ is nonnegative definite if and only if $A_{11}$ and $A_{22}-A_{12}^{*}A_{11}^{-1}A_{12}$ are nonnegative definite.
\end{lem}
\begin{proof}
If $A_{11}$ is invertible, then we have the following congruent transformation
$$\left(\begin{matrix}I & 0\\
-A_{12}^{*}A_{11}^{-1} & I\\
\end{matrix}\right )\left(\begin{matrix}A_{11} & A_{12}\\
A_{12}^{*} & A_{22}\\
\end{matrix}\right )\left(\begin{matrix}I & -A_{11}^{-1}A_{12}\\
0 & I\\
\end{matrix}\right )=\left(\begin{matrix}A_{11} & 0\\
0 & A_{22}-A_{12}^{*}A_{11}^{-1}A_{12}\\
\end{matrix}\right ).$$
Hence, $\left(\begin{matrix}A_{11} & 0\\
0 & A_{22}-A_{12}^{*}A_{11}^{-1}A_{12}\\
\end{matrix}\right )$ is nonnegative definite if and only if $A_{11}$ and $A_{22}-A_{12}^{*}A_{11}^{-1}A_{12}$ are nonnegative definite. We know the congruent transformation does not change the nonnegative definiteness of the matrix. This completes the proof.
\end{proof}

In the following, we will introduce some concepts and give the preliminary results extracted from \cite{AE,JiSar}, which will be used later.
\begin{defn}\cite{JiSar}\label{1.27}
Let $\mathcal{H}_{K}$ be a Hilbert space with reproducing kernel $K:\mathbb{D}\times\mathbb{D}\rightarrow\mathbb{C}$ and the multiplication operator $M_z$ on $\mathcal{H}_{K}$ is bounded.
Suppose that $M_z^*\in B_{1}(\mathbb{D})$.
We say $M_z^*$ is a Cowen-Douglas atom if
\begin{enumerate}
  \item [(1)] the set of polynomials $\mathbb{C}[z]$ is dense in $\mathcal{H}_{K}$;
  \item [(2)] there exists a sequence of polynomials $\{p_{l}(z, \overline{\omega})\}_{l}\subseteq \mathbb{C}[z, \overline{\omega}]$ such that $p_{l}(z, \overline{\omega})\rightarrow\frac{1}{K(z, \omega)}$ as $l\rightarrow \infty$, for all $z, \omega\in \mathbb{D}$;
  \item [(3)] $\sup\limits_{l}\|p_{l}(M_z, M^{*}_z)\|<\infty$;
  \item [(4)] ${\mathcal A}^{\prime}(M_z)=\{M_{\varphi}: \varphi\in H^{\infty}(\mathbb{D})\}$.
\end{enumerate}
\end{defn}

There are many operators in the class of Cowen-Douglas atoms, such as the adjoint of multiplication operators in Hardy space and weighted Bergman space.
Next is the definition of M-contractive given by the first author and J. Sarkar in \cite{JiSar}, which is an analogue of contractive.
\begin{defn}\cite{JiSar}
Let $M$ is a Cowen-Douglas atom with the sequence of polynomials $\{p_{l}(z, \overline{\omega})\}$ as in (2) of Definition \ref{1.27}.
Operator $T\in\mathcal{L}(\mathcal{H})$ is said to be $M$-contractive if
\begin{enumerate}
  \item [(1)] $\sup\limits_{l}\|p_{l}(T^{*}, T)\|<\infty$;
  \item [(2)] $WOT-\lim\limits_{l\rightarrow\infty}p_{l}(T^{*}, T)$ is a positive operator.
\end{enumerate}
\end{defn}

Based on the concept of $M$-contractive, there is the following form of the model theorem due to J. Arazy and M. Engli$\check{s}$ (\cite{AE}, Corollary 3.2).
\begin{thm}[{Model Theorem}, \cite{AE}]\label{1.282}
Let $T\in B_{n}(\mathbb{D})$. Let $M_z$ be the multiplication operator on the Hilbert space $\mathcal{H}_{K}$with reproducing kernel $K$ and $M_z^*$ be a Cowen-Douglas atom.
Then $T$ is $M_z^*$-contractive if and only if $T\sim_{u}M_{z,E}^*|_{K}$, where $E$ is a Hilbert space and $K\subset\mathcal{H}_{K}\otimes E$ is an invariant subspace of $M_{z,E}^*$.
\end{thm}

Given $\mathbb{C}^n$- and $\mathbb{C}^m$-valued reproducing kernel Hilbert spaces $\mathcal{H}$ and $\widetilde{\mathcal{H}}$, respectively, over
the domain $\Omega\subset\mathbb{C}$, a function $\phi: \Omega\rightarrow \mathcal{L}(\mathbb{C}^n, \mathbb{C}^m)$ is said to be a multiplier if $\phi f\in\widetilde{\mathcal{H}}$ for all $f\in\mathcal{H}$,
where $(\phi f)(z)=\phi(z)f(z)$ for all $z\in\Omega$. The set of all such multipliers is denoted $\mathcal{M}(\mathcal{H},\widetilde{\mathcal{H}})$. If $\mathcal{H}=\widetilde{\mathcal{H}}$,
then $\mathcal{M}(\mathcal{H},\widetilde{\mathcal{H}})=\mathcal{M}(\mathcal{H})$.
By the closed graph theorem, each $\phi\in \mathcal{M}(\mathcal{H},\widetilde{\mathcal{H}})$ induces a bounded
linear map $M_\phi:\mathcal{H}\rightarrow\widetilde{\mathcal{H}}$. Consequently, $\mathcal{M}(\mathcal{H},\widetilde{\mathcal{H}})$ is a Banach space with
$\|\phi\|_{\mathcal{M}(\mathcal{H},\widetilde{\mathcal{H}})}=\|M_{\phi}\|_{\mathcal{L}(\mathcal{H},\widetilde{\mathcal{H}})}. $
It follows from Proposition 2.4 of \cite{RS2} that $\phi\in\mathcal{M}(\mathcal{H})$ is equivalent to $M_\phi M_z=M_z M_\phi$, where $M_z$ is the multiplication operator acting on $\mathcal{H}$.
By the condition (4) of Definition \ref{1.27}, we know that $\mathcal{A}'(M_z)\cong H^\infty(\mathbb{D})$ and $\mathcal{M}(\mathcal{H}_K)=H^\infty(\mathbb{D})$.

According to the model theorem \ref{1.282},
and influenced by the similarity theorem given by R. G. Douglas, S. Treil, H. Kwon, Y. Hou and the first author in \cite{DKS, HJK} (Theorem \ref{DKS}), we have the following lemma.
\begin{lem}\label{9.4}
Let $T\in B_{n}(\mathbb{D})$. Let $M$ be a Cowen-Douglas atom on some analytic functional Hilbert space $\mathcal{H}_{K}$ with reproducing kernel $K$.
If $T$ is M-contractive, then $T\sim_{s}M^{(n)}$ if and only if there is a bounded subharmonic function $\varphi$ over $\mathbb{D}$ such that
$$trace\ \mathcal{K}_{M^{(n)}}(w)-trace\ \mathcal{K}_{T}(w)=\frac{\partial ^{2}}{\partial w\partial\bar{w} }\varphi(w),\ w\in\mathbb{D}. $$
\end{lem}
\begin{thm}\label{xiangsi}
Let $T\in B_{n}(\mathbb{D})$. Let $M$ be a Cowen-Douglas atom on some analytic functional Hilbert space $\mathcal{H}_{K}$ with reproducing kernel $K$.
Suppose that there exist a metric $h_T$ of $E_T$ and a constant $C_{1}>0$ such that $\frac{det\,h_{T}(\omega)}{K^{n}(\overline{\omega}, \overline{\omega})}<C_{1}$ for any $\omega\in \mathbb{D}$.
If there exists $\mathcal{M}\in Lat(T)$ such that
\begin{enumerate}
  \item [(1)] $T|_{\mathcal{M}}\in B_{m}(\mathbb{D}), m<n$;
  \item [(2)] $T|_{\mathcal{M}}$ and $(T^{*}|_{\mathcal{M}^{\perp}})^{*}$ are $M$-contractive;
  \item [(3)] $\inf\big\{\frac{det\,h_{T|_{\mathcal{M}}}(\omega)\cdot det\,h_{(T^{*}|_{\mathcal{M}^{\perp}})^{*}}(\omega)}{K^{n}(\overline{\omega}, \overline{\omega})},w\in\mathbb{D}\big\}>0$,
\end{enumerate}
then $T^{*}|_{\mathcal{M}^{\perp}}\sim_{s}M^{*(n-m)}$.
\end{thm}
\begin{proof}
Since $\mathcal{M}\in Lat(T)$, $T$ has the following form for the space decomposition of $\mathcal{H}=\mathcal{M}\oplus\mathcal{M}^{\perp}$:
$$T=\left(\begin{matrix}
T|_{\mathcal{M}} & P_{\mathcal{M}}T|_{\mathcal{M}^{\perp}}\\
0 & P_{\mathcal{M}^{\perp}}T|_{\mathcal{M}^{\perp}} \\
\end{matrix}\right)\begin{matrix}
\mathcal{M} \\
\mathcal{M}^{\perp} \\
\end{matrix},$$
where $T|_{\mathcal{M}}$ denotes the operator whose $T$ is restricted to $\mathcal{M}$ and
$P_{\mathcal{M}}, P_{\mathcal{M}^{\perp}}$ denote the orthogonal projection onto $\mathcal{M}$ and $\mathcal{M}^{\perp}$, respectively.
Similarly, we have
$$T^{*}=\left(\begin{matrix}
(T|_{\mathcal{M}})^{*} & 0\\
(P_{\mathcal{M}}T|_{\mathcal{M}^{\perp}})^{*} & (P_{\mathcal{M}^{\perp}}T|_{\mathcal{M}^{\perp}})^{*} \\
\end{matrix}\right )\begin{matrix}
\mathcal{M} \\
\mathcal{M}^{\perp} \\
\end{matrix}=\left(\begin{matrix}
P_{\mathcal{M}}T^{*}|_{\mathcal{M}} & P_{\mathcal{M}}T^{*}|_{\mathcal{M}^{\perp}}\\
P_{\mathcal{M}^{\perp}}T^{*}|_{\mathcal{M}} & P_{\mathcal{M}^{\perp}}T^{*}|_{\mathcal{M}^{\perp}} \\
\end{matrix}\right)\begin{matrix}
\mathcal{M} \\
\mathcal{M}^{\perp} \\
\end{matrix}. $$
It follows that $P_{\mathcal{M}}T^{*}|_{\mathcal{M}^{\perp}}=0$ and $\mathcal{M}^{\perp}\in Lat(T^{*})$. Thus, $$T=\left(\begin{matrix}
T|_{\mathcal{M}} & P_{\mathcal{M}}T|_{\mathcal{M}^{\perp}}\\
0 & (T^{*}|_{\mathcal{M}^{\perp}})^{*} \\
\end{matrix}\right)\begin{matrix}
\mathcal{M} \\
\mathcal{M}^{\perp} \\
\end{matrix}:=\left(\begin{matrix}
T_{0} & T_{01}\\
0 & T_{1} \\
\end{matrix}\right)\begin{matrix}
\mathcal{M} \\
\mathcal{M}^{\perp} \\
\end{matrix}.$$
Since $T\in B_{n}(\mathbb{D})$ and $T|_{\mathcal{M}}\in B_{m}(\mathbb{D})$, by Lemma \ref{yin1}, we obtain $(T^{*}|_{\mathcal{M}^{\perp}})^{*}\in B_{n-m}(\mathbb{D})$.
We can find a holomorphic frames $\{\gamma_{i}\}_{i=1}^{n}$ of Hermitian holomorphic vector bundle $E_{T}$ from Lemma \ref{yin2} in the form
$\gamma_{i}=(\gamma_{i, 1}, \cdots, \gamma_{i, i}, \underbrace{0,0,\cdots,0 }_{\text{n-i}})^{T}$ and $\gamma_{i, i}\neq0$. Then the Gramian metric $h_T$ of $E_{T}$ is
$$h_{T}=\left(\begin{matrix}\|\gamma_{11}\|^{2} & \langle\gamma_{2, 1}, \gamma_{1, 1}\rangle & \cdots & \langle\gamma_{n, 1}, \gamma_{1, 1}\rangle\\
\langle\gamma_{1, 1}, \gamma_{2, 1}\rangle & \|\gamma_{21}\|^{2}+\|\gamma_{22}\|^{2} & \cdots & \langle\gamma_{n, 1}, \gamma_{2, 1}\rangle+\langle\gamma_{n, 2}, \gamma_{2, 2}\rangle\\
\vdots&\vdots & \ddots & \vdots\\ \langle\gamma_{1, 1}, \gamma_{n, 1}\rangle & \langle\gamma_{2, 1}, \gamma_{n, 1}\rangle+\langle\gamma_{2, 2}, \gamma_{n, 2}\rangle & \cdots & \|\gamma_{n, 1}\|^{2}+\cdots+\|\gamma_{n,n}\|^{2}\\
\end{matrix}\right ). $$
Next, we divide $h_{T}$ into $\left(\begin{matrix}
A & B\\
C & D\\
\end{matrix}\right )$ such that $$A=\left(\begin{matrix}\|\gamma_{11}\|^{2} & \langle\gamma_{2, 1}, \gamma_{1, 1}\rangle & \cdots & \langle\gamma_{m, 1}, \gamma_{1, 1}\rangle\\
\langle\gamma_{1, 1}, \gamma_{2, 1}\rangle & \|\gamma_{21}\|^{2}+\|\gamma_{22}\|^{2} & \cdots & \langle\gamma_{m, 1}, \gamma_{2, 1}\rangle+\langle\gamma_{m, 2}, \gamma_{2, 2}\rangle\\
\vdots&\vdots & \ddots & \vdots\\ \langle\gamma_{1, 1}, \gamma_{m, 1}\rangle & \langle\gamma_{2, 1}, \gamma_{m, 1}\rangle+\langle\gamma_{2, 2}, \gamma_{m, 2}\rangle & \cdots & \|\gamma_{m, 1}\|^{2}+\cdots+\|\gamma_{m, m}\|^{2}\\
\end{matrix}\right ). $$
Then the corresponding parts of $B, C$ and $D$ can be obtained.
By the proof process of Lemma \ref{yin2}, we know $A=h_{T|_{\mathcal{M}}}$ and
$$h_{(T^{*}|_{\mathcal{M}^{\perp}})^{*}}=\left(\begin{matrix}\|\gamma_{m+1,m+1}\|^{2} &  \cdots & \langle\gamma_{n,m+1}, \gamma_{m+1,m+1}\rangle\\
\langle\gamma_{m+1,m+1}, \gamma_{m+2,m+1}\rangle & \cdots & \langle\gamma_{n,m+1}, \gamma_{m+2,m+1}\rangle+\langle\gamma_{n,m+2}, \gamma_{m+2,m+2}\rangle\\
\vdots&  & \vdots\\ \langle\gamma_{m+1,m+1}, \gamma_{n,m+1}\rangle & \cdots & \|\gamma_{n,m+1}\|^{2}+\cdots+\|\gamma_{n, n}\|^{2}\\
\end{matrix}\right).$$
It follows that there exists $D_{2}$, such that $D=h_{(T^{*}|_{\mathcal{M}^{\perp}})^{*}}+D_{2}$.

It is well known that for any $\left(\begin{matrix}M_{1} & M_{2}\\
M_{3} & M_{4}\\
\end{matrix}\right)$,
when $M_{1}$ is invertible, $$det \left(\begin{matrix}M_{1} & M_{2}\\
M_{3} & M_{4}\\
\end{matrix}\right)=det(M_{1})\cdot det(M_{4}-M_{3}M_{1}^{-1}M_{2}). $$ Note that the Gramian matrix $h_{T|_{\mathcal{M}}}$ is invertible, so
\begin{equation}\label{9.41}
\begin{array}{lll}
det\,h_{T}
&=&det\,(h_{T|_{\mathcal{M}}})\cdot det\,[D-C(T|_{\mathcal{M}})^{-1}B]\\
&=&det\,(h_{T|_{\mathcal{M}}})\cdot det\,[h_{(T^{*}|_{\mathcal{M}^{\perp}})^{*}}+D_{2}-C(T|_{\mathcal{M}})^{-1}B].
\end{array}
\end{equation}
Define a new vector set as the following: $$\widehat{\gamma}_{i}:=
\left\{
\begin{array}{cc}
\gamma_{i}, 1\leq i\leq m;\\
(\gamma_{i, 1}, \cdots, \gamma_{i, m}, \underbrace{0,0,\cdots,0 }_{\text{n-m}})^{T}, m+1\leq i\leq n.
\end{array}
\right .$$  The Gramian matrix corresponding to vectors $\{\widehat{\gamma}_{i}\}_{i=1}^{n}$ is
$\left(\begin{matrix}
A & B\\
C & D_{2}\\
\end{matrix}\right )$, which is nonnegative definite. From the invertibility of Gramian matrix $A=h_{T|_{\mathcal{M}}}$ and Lemma \ref{9.2}, we know $D_{2}-C(T^{*}|_{\mathcal{M}})^{-1}B$ is also nonnegative definite.

For the nonnegative definite matrices $N_{1}, N_{2}$, we have $det(N_{1}+N_{2})\geq det(N_{1})+det(N_{2})$. Then equation (\ref{9.41}) can imply that
$$\begin{array}{lll}
det\,h_{T}
&\geq&det(h_{T|_{\mathcal{M}}})\cdot[det(h_{(T^{*}|_{\mathcal{M}^{\perp}})^{*}})+det(D_{2}-C(h_{T|_{\mathcal{M}}})^{-1}B)]\\
&=&det(h_{T|_{\mathcal{M}}})\cdot det(h_{(T^{*}|_{\mathcal{M}^{\perp}})^{*}})+det(h_{T|_{\mathcal{M}}})\cdot det[D_{2}-C(h_{T|_{\mathcal{M}}})^{-1}B]\\
&\geq&det(h_{T|_{\mathcal{M}}})\cdot det(h_{(T^{*}|_{\mathcal{M}^{\perp}})^{*}})\\
&=&det(h_{T|_{\mathcal{M}}\oplus (T^{*}|_{\mathcal{M}^{\perp}})^{*}}).
\end{array}$$
Setting $\widetilde{T}=T|_{\mathcal{M}}\oplus (T^{*}|_{\mathcal{M}^{\perp}})^{*}$. We have $$\frac{det\,h_{\widetilde{T}}(\omega)}{det\,h_{M^{(n)}}(\omega)}=\frac{det\, h_{\widetilde{T}}(\omega)}{K^{n}(\overline{\omega}, \overline{\omega})}\leq\frac{det\,h_{T}(\omega)}{K^{n}(\overline{\omega}, \overline{\omega})}<C_{1}, \omega\in \mathbb{D}. $$
For any $\omega\in \mathbb{D}$, by using the positive definiteness of invertible Gramian matrices $h_{T|_{\mathcal{M}}}(\omega),h_{(T^{*}|_{\mathcal{M}^{\perp}})^{*}}(\omega)$ and $K(\overline{\omega},\overline{\omega})$, we have $\frac{det\,h_{\widetilde{T}}(\omega)}{K^{n}(\overline{\omega}, \overline{\omega})}>0$.
It follows immediately that $\inf\big\{\frac{det\,h_{\widetilde{T}}(\omega)}{K^{n}(\overline{\omega}, \overline{\omega})},w\in\mathbb{D}\big\}>0$, which is clear from condition (3) in the theorem.

By hypothesis (2) in the Theorem, we have $T|_{\mathcal{M}}, (T^{*}|_{\mathcal{M}^{\perp}})^{*}$ are $M-$contractive operators, so is $\widetilde{T}$.
That means vector bundle $E_{\widetilde{T}}$ has the tensor structure $E_M\bigotimes V$ modulo unitary equivalence for some rank $n$ Hermitian holomorphic vector bundle $V$ over $\mathbb{D}$.
Let $\varphi(w):=\log\frac{det\,h_{\widetilde{T}}(w)}{det\,h_{M^{(n)}}(w)}$, $\omega\in \mathbb{D}$. Then $\varphi$ is a bounded subharmonic function acting on $\mathbb{D}$ and
$$trace\mathcal{K}_{M^{(n)}}(w)-trace\mathcal{K}_{\widetilde{T}}(w)=\frac{\partial ^{2}}{\partial w\partial\bar{w} }\varphi(w),\ \omega\in \mathbb{D}. $$
By Lemma \ref{9.4}, we obtain $\widetilde{T}\sim_{s}M^{(n)}$.

It is proved by C. K. Fong and C. L. Jiang in \cite{FJ} that $B_{1}(\Omega)\subset(SI)$.
Since $M\in B_{1}(\Omega)$, by Theorem \ref{12.13}, we know $K_0({\mathcal A}^{\prime}(M) \cong\mathbb{Z}$. For any positive integer $i$, it follows from Theorem \ref{9.14} in \cite{Jiang} that
\begin{equation}\label{9.142}
K_0({\mathcal A}^{\prime}(M^{(i)}))=K_0(M_{i}(\mathbb{C})\otimes{\mathcal A}^{\prime}(M))\cong\mathbb{Z}.
\end{equation}
By Theorem \ref{9.14} and equation (\ref{9.142}) above, it is easy to see that
\begin{equation}\label{9.133}
K_0({\mathcal A}^{\prime}(T|_{\mathcal{M}}\oplus (T^{*}|_{\mathcal{M}^{\perp}})^{*}\oplus M^{(n)}))=K_0({\mathcal A}^{\prime}(M^{(n)}))\cong\mathbb{Z},
\end{equation}
since $T|_{\mathcal{M}}\oplus (T^{*}|_{\mathcal{M}^{\perp}})^{*}\sim_{s}M^{(n)}$.
We know that each Cowen-Douglas operator can be written as a direct sum of a finite number of strongly irreducible Cowen-Douglas operators up to similarity.
Without loss of generality, we assume that
$$T|_{\mathcal{M}}=\bigoplus\limits_{i=1}^{t}A_{i}^{(m_{i})}, A_{i}\in B_{n_{i}}(\mathbb{D})\cap(SI), A_{i}\nsim_{s}A_{j}(i\neq j), \sum\limits_{i=1}^{t}m_{i}n_{i}=m, m_{i}>0;$$
$$(T^{*}|_{\mathcal{M}^{\perp}})^{*}=\bigoplus\limits_{i=1}^{s}B_{i}^{(l_{i})}, B_{i}\in B_{h_{i}}(\mathbb{D})\cap(SI), B_{i}\nsim_{s}B_{j}(i\neq j), \sum\limits_{i=1}^{s}h_{i}l_{i}=n, l_{i}>0.$$

Claim: $A_i\sim_s M^{(n_i)}$ and $B_j\sim_s M^{(h_j)}$ for any $1\leq i\leq t$ and $1\leq j\leq s$.

Otherwise, there exist $t'\leq t,\,s'\leq s$ (at least one of $t'$ and $s'$ is greater than 1) such that $A_{i}\nsim_{s}M^{(n_{i})}$ and $B_{i}\nsim_{s}M^{(h_{i})}$ when $1\leq i\leq t^{'}$ and $1\leq i\leq s^{'}$ (If we rearrange $\{A_i\}_{i=1}^t$ and $\{B_j\}_{j=1}^s$, this will always work).
Then $\bigoplus\limits_{i=1}^{t^{'}}A_{i}^{(m_{i})}\bigoplus M^{\big(\sum\limits_{i=1}^{t'}n_im_i\big)}$ is a finite (SI) decomposition, by Theorem \ref{9.52}, we know that
$$K_0\biggl({\mathcal A}^{\prime}\Big(\bigoplus\limits_{i=1}^{t^{'}}A_{i}^{(m_{i})}\bigoplus M^{\big(\sum\limits_{i=1}^{t'}n_im_i\big)}\Big)\biggl)\cong\mathbb{Z}^{t^{'}+1}. $$
Similarly, we have
$$K_0\biggl({\mathcal A}^{\prime}\Big(\bigoplus\limits_{j=1}^{s^{'}}B_{j}^{(l_{j})}\bigoplus M^{\big(\sum\limits_{j=1}^{s'}h_jl_j\big)}\Big)\biggl)\cong\mathbb{Z}^{s^{'}+1}. $$
Note that $$\bigoplus\limits_{i=1}^{t^{'}}A_{i}^{(m_{i})}\bigoplus\limits_{i=1}^{s^{'}}B_{i}^{(l_{i})}\bigoplus M^{(k)}\sim_{s}\bigoplus\limits_{i=1}^{t^{'}}A_{i}^{(m_{i})}\bigoplus M^{\big(\sum\limits_{i=1}^{t'}n_im_i\big)}\bigoplus\limits_{j=1}^{s^{'}}B_{j}^{(l_{j})}\bigoplus M^{\big(\sum\limits_{j=1}^{s'}h_jl_j\big)}\bigoplus M^{(k')}, $$
where $k=n+\sum\limits_{i=t^{'}+1}^{t}m_{i}n_{i}+\sum\limits_{i=s^{'}+1}^{s}h_{i}l_{i}$ and $k'=2\sum\limits_{i=t^{'}+1}^{t}m_{i}n_{i}+2\sum\limits_{i=s^{'}+1}^{s}h_{i}l_{i}$.
Hence, we have
$$K_0\biggl({\mathcal A}^{\prime}\biggl(\bigoplus\limits_{i=1}^{t}A_{i}^{(m_{i})}\bigoplus\limits_{i=1}^{s}B_{i}^{(l_{i})}\bigoplus M^{(n)}\biggl)\biggl)
\cong\mathbb{Z}^{r+1},$$
where $1\leq \max\{t^{'},s^{'}\}\leq r\leq t^{'}+s^{'}$. Combining with the equation (\ref{9.133}), we obtain that this is a contradiction. Hence, $t^{'}=s^{'}=0$ and
for any $1\leq i\leq t$, $1\leq j\leq s$, we have $A_{i}\sim_{s}M^{(n_{i})}, B_{j}\sim_{s}M^{(h_{j})}$. Since $(T^{*}|_{\mathcal{M}^{\perp}})^{*}\in B_{n-m}(\mathbb{D})$, we infer that  $T^{*}|_{\mathcal{M}^{\perp}}\sim_{s}M^{*(n-m)}$.
\end{proof}

For a Cowen-Douglas atom $M$ on Hilbert space $\mathcal{H}_K$ with reproducing kernel $K$, we know that there exists a sequence of polynomials $\{p_{l}(z, \overline{\omega})\}_{l}\subseteq \mathbb{C}[z, \overline{\omega}]$ such that $p_{l}(z, \overline{\omega})\rightarrow\frac{1}{K(z, \omega)}$ as $l\rightarrow \infty$, for all $z, \omega\in \mathbb{D}$.
If $T\in\mathcal{L}(\mathcal{H})$ is $M$-contractive, then $C:=WOT-\lim\limits_{l\rightarrow\infty}p_{l}(T^{*}, T)$ is positive. Let $D=C^{\frac{1}{2}}$ and $E=\overline{D\;\mathcal{H}}$.
Subsequently, there exist an isometry $V:\mathcal{H}\rightarrow K\subset\mathcal{H}_K\bigotimes E$ defined by $Vh=\sum\limits_{i=0}^\infty e_i(\cdot)\bigotimes De_i(T^*)^*h$
for an orthonormal basis $\{e_i\}_{i=0}^\infty$ of $\mathcal{H}_K$, $h\in\mathcal{H}$ and $K=ran\;V$.
It follows that for any $t(w)\in\ker(T-w)$, $Vt(w)=K(\cdot,\overline{w})\bigotimes Dt(w)$.
Let $T\in B_n(\mathbb{D})$. Choosing a holomorphic frame $\{\gamma_i\}_{i=1}^n$ of $E_T$, we have $E_T\sim_uE_M\bigotimes \mathcal{E}$, where a frame of bundle $\mathcal{E}$ is $\{D\gamma_i\}_{i=1}^n$.
Then there exists a metric of $E_T$ is
$h_T(w)=\big(\langle\gamma_j(w),\gamma_i(w)\rangle\big)_{n\times n}=K(\overline{w},\overline{w})I_{n\times n}\big(\langle D\gamma_j(w),D\gamma_i(w)\rangle\big)_{n\times n}$.
Thus, $det\,h_T(w)=K^n(\overline{w},\overline{w})det\,\big(\langle D\gamma_j(w),D\gamma_i(w)\rangle\big)_{n\times n}=K^n(\overline{w},\overline{w})det\,h_{\mathcal{E}}(w)$.
That means the hypothesis (3) of Theorem \ref{xiangsi} is $\inf\big\{det\,h_{\mathcal{E}_1}(w)det\,h_{\mathcal{E}_2}(w),w\in\mathbb{D}\big\}>0$,
where $E_{T|_{\mathcal{M}}}\sim_uE_M\bigotimes \mathcal{E}_1$ and $E_{(T^{*}|_{\mathcal{M}^{\perp}})^{*}}\sim_uE_M\bigotimes \mathcal{E}_2$.

It is well known that a necessary condition for two Cowen-Douglas operators of index one to be similar is that the ratio of the metrics of the operators is bounded and bounded from zero.
So condition (3) in Theorem \ref{xiangsi} describes a natural sufficient condition for operator similarity.
It has also been proved to be a necessary and sufficient condition when the operator is $n$-hypercontractive and similar to a weighted Bergman shift, $n\geq1$.
The following remark, we will check that it also naturally holds in many other cases.

\begin{rem}
Let $\mathcal{H}_K$ be a analytic functions Hilbert space with reproducing kernel $K:\Omega\times\Omega\rightarrow \mathbb{C}$. Suppose that $\mathcal{H}_K$
contains the constant functions and the polynomials are dense in $\mathcal{H}_K$, $K(z,w)\neq0$ for all $z,w\in\Omega$ and $\frac{1}{K}$ is a polynomial.
Let the multiplication operator $M_z$ is bounded on the analytic functional Hilbert space $\mathcal{H}_{K_1}$ with reproducing kernel $K_{1}(z,\omega)=\sum\limits_{n=0}^{\infty}b_{n}z^{n}\overline{\omega}^{n}$, $z,w\in\Omega$, $b_{n}>0$ and $\sigma(M_z)\subset\Omega$.
Suppose that $\frac{1}{K}(z,w)=\sum\limits_{n=0}^ka_nz^n\bar{w}^n$ and $a_0>0$.
If $\frac{1}{K}(M_z,M_z^*)\geq0$ (This is a contractive condition for a new model theorem, which is a generalization of the $n$-hypercontration), then we have
\begin{equation}\label{202303211}
\begin{cases}
a_0\geq0,\\
a_0+\sum\limits_{i=1}^ma_i\frac{b_{m-i}}{b_m}\geq0,\,\,1 \leq m \leq k,\\
a_0+\sum\limits_{i=1}^ka_i\frac{b_{l-i}}{b_l}\geq0,\,\,l> k.
\end{cases}
\end{equation}
Note that $\frac{K_{1}(\overline{\omega},\overline{\omega})}{K(\overline{\omega},\overline{\omega})}=\sum\limits_{m=0}^k\big(\sum\limits_{i=0}^ma_ib_{m-i}\big)|w|^{2m}+\sum\limits_{l=k+1}^\infty\big(\sum\limits_{i=0}^ka_ib_{l-i}\big)|w|^{2l}$.
Since equation (\ref{202303211}), we know that the coefficients of $|w|^{2n}$, $n\geq0$ are all non negative.
Then $\frac{K_{1}(\overline{\omega},\overline{\omega})}{K(\overline{\omega},\overline{\omega})}\geq a_0b_0>0$ for all $w\in\Omega$.

In general, let $M:=(M_z^*,\mathcal{H}_K)$ be a Cowen-Douglas atom and $(M_z^*,\mathcal{H}_{K_1})$ is $M-$contractive.
Then there exist $\{p_{l}\}_{l\geq0}$ such that $p_{l}(z, \overline{\omega})=\sum\limits_{n=0}^{k_{l}}a_{n}^{l}z^{n}\overline{\omega}^{n}$ and $WOT-\lim\limits_{l\rightarrow\infty}p_{l}(M_z,M_z^{*})\geq0$, where $M_z$ is on space $\mathcal{H}_{K_1}$.
A routine calculation shows that for any positive integer $l$, $p_l(M_z,M_z^{*})=diag(c_0^l,c_1^l,c_2^l,\cdots)$,
where
$$c_i^l=
\begin{cases}
a_0^l+\sum\limits_{j=1}^ia_j^l\frac{b_{i-j}}{b_i},\,\,1 \leq i \leq k_l,\\
a_0^l+\sum\limits_{j=1}^{k_l}a_j^l\frac{b_{i-j}}{b_i},\,\,i >k_l.
\end{cases}$$
That means $\lim\limits_{l\rightarrow\infty}a_0^l\geq0$ and $\lim\limits_{l\rightarrow\infty}c_i^l\geq0$ for any $i\geq1$.
Note that
$$\begin{array}{lll}
\frac{K_{1}(\overline{\omega},\overline{\omega})}{K(\overline{\omega},\overline{\omega})}
&=&\lim\limits_{l\rightarrow\infty}p_{l}(\omega, \overline{\omega})K_{1}(\overline{\omega},\overline{\omega})\\
&=&\lim\limits_{l\rightarrow\infty}(\sum\limits_{n=0}^{k^{(l)}}a_{n}^{l}|\omega|^{2n})(\sum\limits_{n=0}^{\infty}b_{n}|\omega|^{2n})\\
&=&\lim\limits_{l\rightarrow\infty}[a_{0}^{l}b_{0}+\sum\limits_{i=1}^{k^{(l)}}c_i^lb_i|\omega|^{2i}+\sum\limits_{j=k^{(l)}+1}^{\infty}c_i^lb_j|\omega|^{2j}]\\
&=&\lim\limits_{l\rightarrow\infty}a_{0}^{l}b_{0}+\lim\limits_{l\rightarrow\infty}\sum\limits_{i=1}^{k^{(l)}}c_i^lb_i|\omega|^{2i}+\lim\limits_{l\rightarrow\infty}\sum\limits_{j=k^{(l)}+1}^{\infty}c_i^lb_j|\omega|^{2j}\\
&\geq&\lim\limits_{l\rightarrow\infty}a_{0}^{l}b_{0}.
\end{array}$$
If $\lim\limits_{l\rightarrow\infty}a_0^l>0$, then $\inf\big\{\frac{K_{1}(\overline{\omega},\overline{\omega})}{K(\overline{\omega},\overline{\omega})}, w\in\mathbb{D}\big\}>0$.

The above examples are discussed in order to state the boundedness of the ratio of metrics corresponding to line bundles. Indeed, many bundles with rank two contain reducible and irreducible cases, which can be generated from line bundles.
Given a Hermitian holomorphic line bundle $\mathcal{E}$ over a bounded domain $\Omega\subset\mathbb{C}$.
Let $\gamma$ be a holomorphic frame of $\mathcal{E}$ and $h$ be a metric of $\mathcal{E}$ corresponding to $\gamma$.
Since $\gamma(w)$ and $\gamma'(w)$ are linearly independent for any $w\in\Omega$, we know that $\{\gamma,\gamma'\}$ is a frame of jet bundle $\mathcal{J}_2(\mathcal{E})$.
Then the form of a Gram metric matrix of the bundle $\mathcal{J}_2(\mathcal{E})$ is $\mathcal{J}_2(h)=\left(\begin{matrix} h & \partial h\\
\overline{\partial} h &\partial\overline{\partial} h \\
\end{matrix}\right)$ and $det\,\mathcal{J}_2(h)=h^2(-\mathcal{K})$, where $\mathcal{K}$ is the curvature of $\mathcal{E}$.
Let $K(z,w)=\langle \gamma(\overline{w}),\gamma(\overline{z})\rangle$ for $z,w\in\Omega$.
If $K$ is infinitely divisible, then there exists a domain $\Omega_0\subseteq\Omega$ such that the curvature
matrix $\mathcal{K}=\frac{\partial^2}{\partial w\overline{\partial}w}\log K$ is a positive-definite function on $\Omega_0$ due to S. Biswas, G. Misra and the third author in \cite{BKM}.
Thus, $det\,\mathcal{J}_2(h)$ is positive definite over $\Omega_0$.
For kernel $K$, there exists a unique Hilbert space $\mathcal{H}_K$ with reproducing kernel $K$.
Then we may find operators $S$ and $T_1$ satisfy $T:=\left(\begin{matrix} M_z^* & S\\
0 & T_1 \\
\end{matrix}\right)$ belongs to $\mathcal{F}B_2(\Omega_0)$, i.e. $T_1\in B_1(\Omega_0)$ and $M_z^*S=ST_1$, introduced in \cite{JJKM}.
Subsequently, there exists a metric $h_T$ of bundle $E_T$ such that $det\,h_T=det\,\mathcal{J}_2(h)+hh_1=hh_1-h^2\mathcal{K}$ is also a positive kernel over $\Omega_0$, where $h_1$ is a metric of $E_{T_1}$.
If the adjoint of multiplication operator $\widetilde{M}_z$ determined by kernel $det\,h_T$ is $n$-hypercontractive and $\partial^i\overline{\partial}^jdet\,h_T(w_0)=0$ when $i\neq j$ for some $w_0\in\Omega_0\subset\mathbb{D}$ and positive integer $n$, then $\inf\{\frac{det\,h_T(w)}{K^{(2n)}(\overline{w},\overline{w})},w\in\Omega_0\}>0$, where $K^{(2n)}(\overline{w},\overline{w})=\frac{1}{(1-|w|^2)^{2n}}$.

In fact, we just need $\inf\{\frac{h(w)h_1(w)}{K^{(2n)}(\overline{w},\overline{w})},w\in\Omega_0\}$ or $\inf\{\frac{-h^2(w)\mathcal{K}(w)}{K^{(2n)}(\overline{w},\overline{w})},w\in\Omega_0\}$ to be greater than zero.
\end{rem}


In Theorem \ref{xiangsi}, when $T\in B_{2}(\mathbb{D})$ and $T|_{\mathcal{M}}\in B_{1}(\mathbb{D})$, we concluded that $T|_{\mathcal{M}}\sim_{s}(T^{*}|_{\mathcal{M}^{\perp}})^{*}\sim_{s}M_{z}^{*}$.
In general, for any
$T=\left(\begin{matrix} T_{1} & T_{12}\\
0&T_{2}\\
\end{matrix}\right )\in B_{2}(\mathbb{D}).$ Suppose that $T\sim_{s}M_{z}^{*(2)}$ for the multiplication operator $M_{z}$ on some analytic functional Hilbert space.
If there exist a diagonal invertible operator intertwining $T$ and $M_{z}^{*(2)}$, by a simple calculation, we verify that $T_{1}\sim_{s}T_{2}\sim_{s}M_{z}^{*}$ and $T_{12}=0$.
If the invertible operator $X$ intertwining $T$ and $M_{z}^{*(2)}$ is upper triangular,
i.e., $X=\left ( \begin{matrix} X_{11} & X_{12}\\
0 & X_{22}\\
\end{matrix}\right )$. By calculation, we obtain $T_{1}\sim_{s}T_{2}\sim_{s}M_{z}^{*}$ and $T_{12}=\tau_{T_{1},T_{2}}(X_{11}^{-1}X_{12})=T_1X_{11}^{-1}X_{12}-X_{11}^{-1}X_{12}T_2$.
This operator class is introduced in \cite{HJJX} and is shown to be strongly reducible.
In general, the invertible operator $X$, which is neither diagonal nor upper triangular, will be more complicated to calculate. The above two cases also show that Theorem \ref{xiangsi} is not-trivial.
\begin{cor}\label{9.7}
Let $T\in B_{1}(\mathbb{D})$. Let $M$ be a Cowen-Douglas atom on some analytic functional Hilbert space $\mathcal{H}_{K}$ with reproducing kernel $K$.
Suppose that there exist a metric $h_T$ of $E_T$ and a constant $C_{1}>0$ such that $\frac{h_{T}(\omega)}{K(\overline{\omega}, \overline{\omega})}<C_{1}$ for any $\omega\in \mathbb{D}$.
If there exists $\mathcal{M}\in Lat(T)$ such that
\begin{enumerate}
  \item [(1)] $\sigma(T|_{\mathcal{M}})\cap \mathbb{D}=\emptyset$;
  \item [(2)] $(T^{*}|_{\mathcal{M}^{\perp}})^{*}$ is $M$-contractive;
  \item [(3)] $\inf\big\{\frac{h_{(T^{*}|_{\mathcal{M}^{\perp}})^{*}}(\omega)}{K(\overline{\omega}, \overline{\omega})},w\in\mathbb{D}\big\}>0$,
\end{enumerate}
then $T^{*}|_{\mathcal{M}^{\perp}}\sim_{s}M^{*}$.
\end{cor}
\begin{proof}
Since $\mathcal{M}\in Lat(T)$, $T$ has the following form under the space decomposition of $\mathcal{M}\oplus\mathcal{M}^{\perp}$:
$$T=\left(\begin{matrix} T|_{\mathcal{M}} & P_{\mathcal{M}}T|_{\mathcal{M}^{\perp}}\\
0 & (T^{*}|_{\mathcal{M}^{\perp}})^{*}\\
\end{matrix}\right ).$$
Let $e=e_{1}+e_{2}$ be a non-zero holomorphic section of vector bundle $E_{T}$. We have for each $w\in\mathbb{D}$
$$\left(\begin{matrix} T|_{\mathcal{M}}-\omega & P_{\mathcal{M}}T|_{\mathcal{M}^{\perp}}\\
0 & (T^{*}|_{\mathcal{M}^{\perp}})^{*}-\omega\\
\end{matrix}\right )\left ( \begin{matrix} e_{1}(\omega)\\
e_{2}(\omega)\\
\end{matrix}\right)=\left ( \begin{matrix} (T|_{\mathcal{M}}-\omega)e_{1}(\omega)+P_{\mathcal{M}}T|_{\mathcal{M}^{\perp}}e_{2}(\omega)\\
((T^{*}|_{\mathcal{M}^{\perp}})^{*}-\omega)e_{2}(\omega)\\
\end{matrix}\right)=\left ( \begin{matrix} 0\\
0\\
\end{matrix}\right).$$
It follows from $\sigma(T|_{\mathcal{M}})\cap \mathbb{D}=\emptyset$ that $e_{2}(\omega)\in \ker((T^{*}|_{\mathcal{M}})^{*}-\omega)$ and $e_{1}(\omega)=-(T|_{\mathcal{M}}-\omega)^{-1}P_{\mathcal{M}}T|_{\mathcal{M}^{\perp}}e_{2}(\omega)$ for every $w\in\mathbb{D}$.
We know that $e_{2}(\omega)$ is not zero holomorphic, otherwise $e(\omega)$ is zero.
That means $\dim \mbox{ker} ((T^{*}|_{\mathcal{M}^{\perp}})^{*}-\omega)=1$ for every $w\in\mathbb{D}$.
Further, it is obtained from $\bigvee \limits_{\omega{\in}{\mathbb{D}}} \ker(T-\omega)=\mathcal{M}\oplus\mathcal{M}^{\perp}$ that $\bigvee \limits_{\omega{\in}{\mathbb{D}}} \ker((T^{*}|_{\mathcal{M}})^{*}-\omega)=\mathcal{M}^{\perp}$.
By $\mbox{ran}(T-\omega)=\mathcal{M}\oplus\mathcal{M}^{\perp}$,
that is, for a fixed but arbitrary $y_{1}+y_{2}\in \mathcal{M}\oplus\mathcal{M}^{\perp}$, there exists $x_{1}+x_{2}\in \mathcal{M}\oplus\mathcal{M}^{\perp}$, such that $$\left(\begin{matrix} T|_{\mathcal{M}}-\omega & P_{\mathcal{M}}T|_{\mathcal{M}^{\perp}}\\
0 & (T^{*}|_{\mathcal{M}^{\perp}})^{*}-\omega\\
\end{matrix}\right )\left ( \begin{matrix} x_{1}\\
x_{2}\\
\end{matrix}\right)=\left ( \begin{matrix} (T|_{\mathcal{M}}-\omega)x_{1}+P_{\mathcal{M}}T|_{\mathcal{M}^{\perp}}x_{2}\\
((T^{*}|_{\mathcal{M}^{\perp}})^{*}-\omega)x_{2}\\
\end{matrix}\right)=\left ( \begin{matrix} y_{1}\\
y_{2}\\
\end{matrix}\right),\,w\in\mathbb{D}.$$ Then~$\mbox{ran}((T^{*}|_{\mathcal{M}^{\perp}})^{*}-\omega)=\mathcal{M}^{\perp}$, $w\in\mathbb{D}$.
Hence, $(T^{*}|_{\mathcal{M}^{\perp}})^{*}\in B_{1}(\mathbb{D})$.

Note that there exist metrics $h_T$ of $E_T$ and $h_{(T^{*}|_{\mathcal{M}^{\perp}})^{*}}$ of $E_{(T^{*}|_{\mathcal{M}^{\perp}})^{*}}$ are
$$\begin{array}{lll}
h_{T}(\omega)
&=&\|e_{1}(\omega)\|^{2}+\|e_{2}(\omega)\|^{2}\\
&=&\|e_{1}(\omega)\|^{2}+h_{(T^{*}|_{\mathcal{M}^{\perp}})^{*}}(\omega)\\
&\geq&h_{(T^{*}|_{\mathcal{M}^{\perp}})^{*}}(\omega),\ \omega{\in}{\mathbb{D}}.
\end{array}$$
From $\frac{h_{T}(\omega)}{K(\overline{\omega},\overline{\omega})}<C_{1}$, we have ~$\frac{h_{(T^{*}|_{\mathcal{M}^{\perp}})^{*}}(\omega)}{K(\overline{\omega},\overline{\omega})}<C_{1}$, $\omega{\in}{\mathbb{D}}$.
By hypotheses (2) and (3), there exists $c_{1}>0$ such that
~$\frac{h_{(T^{*}|_{\mathcal{M}^{\perp}})^{*}}(\omega)}{K(\overline{\omega},\overline{\omega})}>c_{1}$ and
$$\mathcal{K}_{M}(\omega)-\mathcal{K}_{(T^{*}|_{\mathcal{M}^{\perp}})^{*}}(\omega)=\frac{\partial ^{2}}{\partial \omega\partial\bar{\omega} }\varphi(\omega), $$
where $\varphi(\omega):=\log\frac{h_{(T^{*}|_{\mathcal{M}})^{*}}(\omega)}{(1-|\omega|^{2})^{-n}}$ is a bounded subharmonic function on $\mathbb{D}$.
Hence, by Lemma \ref{9.4}, we infer that $T^{*}|_{\mathcal{M}^{\perp}}\sim_{s}M_{z}.$
\end{proof}
The following corollary is a direct consequence of Corollary \ref{9.7}:
\begin{cor}
Let $M$ be a Cowen-Douglas atom on Hilbert space $\mathcal{H}_{K}$ with reproducing kernel $K$.
Let $\mathcal{H}$ be a $\mathbb{C}$-valued Hilbert space over $\mathbb{D}$, determined by the kernel function $\widetilde{K}$.
Suppose that $M_z^*\in B_{1}(\mathbb{D})$.
If there exist $C_{1}>0$ and $I_{i}\in Lat(M_{z})$, $i=1, 2$ such that:
\begin{enumerate}
  \item [(1)] $(M_{I_{i}})^{*}$ is $M-$contractive;
  \item [(2)] $\frac{\widetilde{K}(w, w)}{K(w, w)}\leq C_{1}$ and $\widetilde{K}(\cdot, w)|_{I_{i}}\neq0$ for any $w\in \mathbb{D}$;
  \item [(3)] $\inf\big\{\frac{\|\widetilde{K}(\cdot, w)|_{I_{i}}\|^2}{K(w, w)},w\in\mathbb{D}\big\}>0$,
\end{enumerate}
then $M_{I_{1}}\sim_{s}M_{I_{2}}$.
\end{cor}
\begin{proof}
Since $I_{i}\in Lat(M_{z})$, $M_{z}$ and $M_{z}^{*}$ have the following forms under the space decomposition of $I_{i}^{\perp}\oplus I_{i}$:
$$M_{z}=\left(\begin{matrix} P_{I_{i}^{\perp}}M_{z}|_{I_{i}^{\perp}} & 0\\
P_{I_{i}}M_{z}|_{I_{i}^{\perp}} & M_{I_{i}}\\
\end{matrix}\right),\ M_{z}^{*}=\left(\begin{matrix}
M_{z}^{*}|_{I_{i}^{\perp}} & (P_{I_{i}}M_{z}|_{I_{i}^{\perp}})^{*}\\
0 & (M_{I_{i}})^{*}\\
\end{matrix}\right),\ i=1, 2.$$
Let $e$ be a nonzero holomorphic section of $E_{M_z^*}$ that satisfies $\|e(\omega)\|^{2}=\widetilde{K}(\overline{\omega}, \overline{\omega})$, $w\in\mathbb{D}$.
It follows from $M_{z}^{*}\in B_{1}(\mathbb{D})$ that $\mbox{ran}((M_{I_{i}})^{*}-\omega)=I_{i}$, $e_i(\omega):=e(\omega)|_{I_{i}}\in\ker((M_{I_{i}})^{*}-\omega)$ and $\bigvee \limits_{\omega{\in}{\mathbb{D}}} e_i(\omega)=I_{i}$ for $i=1,2$.
Note that $\widetilde{K}(\cdot, w)|_{I_{i}}\neq0$ for any $w\in \mathbb{D}$, we have $(M_{I_{i}})^{*}\in B_{1}(\mathbb{D})$ and $e_i$ is a nonzero holomorphic section of $E_{(M_{I_{i}})^{*}}$, $i=1,2$.
Since $\frac{\|e_{i}(\omega)\|^{2}}{K(\overline{\omega}, \overline{\omega})}\leq\frac{\|e(\omega)\|^{2}}{K(\overline{\omega}, \overline{\omega})}\leq C_{1}$ for any
$w\in\mathbb{D}$ and hypothesis (3), there exists $c_{1}>0$ such that $c_1\leq\frac{h_{(M|_{I_{i}})^{*}}(\omega)}{K(\overline{\omega}, \overline{\omega})}\leq C_{1}$,
where $h_{(M|_{I_{i}})^{*}}(\omega):=\|e_{i}(\omega)\|^{2}$ is a metric of $E_{(M|_{I_{i}})^{*}}$ for $i=1, 2$.
That means $$\mathcal{K}_{M}(\omega)-\mathcal{K}_{(M_{I_{i}})^{*}}(\omega)=\frac{\partial ^{2}}{\partial \omega\partial\bar{\omega} }\varphi_i(\omega), $$ where $\varphi_i(\omega):=\log\frac{\|e_{i}(\omega)\|^{2}}{K(\overline{\omega}, \overline{\omega})}$ is a bounded subharmonic function, $i=1, 2$.
From Lemma \ref{9.4} and the condition of $(M_{I_{i}})^{*}$ is $M-$contractive, we obtain $M_{I_{i}}\sim_{s}M^{*}, i=1, 2$. Hence, $M_{I_{1}}\sim_{s}M_{I_{2}}$.
\end{proof}

In the following, we will return to the open problem raised by K. Zhu: For the multiplication operator on the Bergman space and $I, J\in Lat(M_{z})$, when does $M_{I}\sim_{s}M_{J}$ hold?
For this problem, a sufficient condition is obtained.
\begin{cor}\label{11.9}
Let $M$ be a Cowen-Douglas atom on Hilbert space $\mathcal{H}_{K}$ with reproducing kernel $K$.
Let $M_{z}$ be the multiplication operator on the Bergman space.
Suppose that there exists $C>0$ such that
$(1-|\omega|^{2})^{2}K(\overline{\omega}, \overline{\omega})\geq C$ for any $\omega\in \mathbb{D}$. If there exist $I_{i}\in Lat(M_{z})$, $i=1, 2$ such that:
\begin{enumerate}
  \item [(1)] $(M_{I_{i}})^{*}\in B_{1}(\mathbb{D})$ and $(M_{I_{i}})^{*}$ are $M-$contractive;
  \item [(2)] $\inf\big\{\frac{h_{(M_{I_{i}})^{*}}(\omega)}{K(\overline{\omega}, \overline{\omega})},w\in\mathbb{D}\big\}>0$,
\end{enumerate}
then $M_{I_{1}}\sim_{s}M_{I_{2}}$.
\end{cor}

\begin{ex}\label{1.28}
Let $M_{z}$ be the multiplication operator on the Bergman space $\mathcal{H}$. Let $e_{n}=\sqrt{n+1}z^{n}$, $\{e_{n}\}_{n=0}^{\infty}$ is an orthonormal basis of Bergman space.
Set $\mathcal{M}:=span\{e_{1}, e_{2}, \cdots\}$. Note taht $$M_{z}\mathcal{M}=span\{e_{2}, e_{3}, \cdots\}\subseteq \mathcal{M},$$ then $\mathcal{M}\in Lat(M_{z})$ and $\{e_{n}\}_{n=1}^{\infty}$ is an orthonormal basis of $\mathcal{M}$.
It is easy to see that $M_{z}$ and $M_{z}|_{\mathcal{M}}$ are unilateral shifts with weight sequences $\big\{\sqrt{\frac{n+1}{n+2}}\big\}_{n=0}^{\infty}$ and $\big\{\sqrt{\frac{n+2}{n+3}}\big\}_{n=0}^{\infty}$ respect to $\{e_{n}\}_{n=0}^{\infty}$ and $\{e_{n}\}_{n=1}^{\infty}$, respectively.
On the one hand, by Theorem \ref{alger1}, we know $M_{z}\sim_{s}M_{z}|_{\mathcal{M}}$.
On the other hand, we know that $M_{z}|_{\mathcal{M}}$ is unitarily equivalent to the multiplication operator on some Hilbert space with reproducing kernel $K_{M_{z}|_{\mathcal{M}}}(z,\omega)=\frac{2-z\overline{\omega}}{(1-z\overline{\omega})^{2}}, z,\omega\in\mathbb{D}$.
Note that $M_{z}|_{\mathcal{M}}$ is not $2$-hypercontration, $\frac{1}{2}<\frac{1}{2-|\omega|^{2}}<1$ for all $w\in\mathbb{D}$ and
$$\mathcal{K}_{(M_{z}|_{\mathcal{M}})^{*}}(\omega)-\mathcal{K}_{M_{z}^{*}}(\omega)=\frac{\partial ^{2}}{\partial w\partial\bar{w} }\log\frac{1}{2-|\omega|^{2}}=\frac{2}{(1-|\omega|^{2})^{2}}>0,\ \omega\in \mathbb{D}.$$

\end{ex}
We know $M_{z}^{*}$ is 2-hypercontractive, $(M_{\mathcal{M}})^{*}$ is contractive. Thus, Theorem \ref{DKS} is invalid for Example \ref{1.28}. In our main theorem, the class of model is expanded. It is not limited to some homogeneous operators, and involves many operators like $(M_{z}|_{\mathcal{M}})^{*}$. So the above example can be verified by Theorem \ref{xiangsi}.
Next, we extend the conclusion of the Corollary \ref{11.9} to the finite direct sum of the multiplication operator on the Bergman space.
\begin{cor}\label{10.211}
Let $M$ be a Cowen-Douglas atom on Hilbert space $\mathcal{H}_{K}$ with reproducing kernel $K$.
Let $M_{z}$ be the multiplication operator on the Bergman space.
Suppose that there exists $C>0$ such that
$(1-|\omega|^{2})^{2}K(\overline{\omega}, \overline{\omega})\geq C$ for any $\omega\in \mathbb{D}$. If $I_{i}\in Lat(M_{z}^{(n)}),\,i=1, 2$ satisfy the
following conditions:
\begin{enumerate}
  \item [(1)] $M_{z}^{*(n)}\mid_{I_{i}^{\perp}}\in B_{m}(\mathbb{D}),\,m<n$;
  \item [(2)] $M_{z}^{*(n)}\mid_{I_{i}^{\perp}}$ and $(M_{z}^{(n)}\mid_{I_{i}})^{*}$ are $M$-contractive;
  \item [(3)] there exists a metric $h_{T_{i}}$ of $E_{T_{i}}$ such that $\inf\{\frac{det\,h_{T_{i}}(\omega)}{K^{n}(\overline{\omega}, \overline{\omega})},\omega\in\mathbb{D}\}>0$, where $T_{i}=M_{z}^{*(n)}|_{I_{i}^{\perp}}\oplus(M_{z}^{(n)}|_{I_{i}})^{*}$,
\end{enumerate}
then $M_{z}^{(n)}\mid_{I_{1}}\sim_{s}M_{z}^{(n)}\mid_{I_{2}}$.
\end{cor}

\begin{ex}
Let $M$ be a Cowen-Douglas atom on Hilbert space $\mathcal{H}_{K}$ with reproducing kernel $K$.
Let $T=\left(\begin{matrix}
M & XM-MX\\
0 & M \\
\end{matrix}\right)$ for some bounded linear operator $X$ on $\mathcal{H}_{K}$.
Set $t(w):=K(\cdot, \overline{w})$, $w\in\mathbb{D}$.
Then $t$ is a non-zero holomorphic section of $E_{M}$. It is easy to verify that $E_{T}=span\{t, X(t)+t\}$.
Then we have
$$h_{T}(\omega)=\left(\begin{matrix}\|t(\omega)\|^{2} & \langle X(t(\omega)),t(\omega)\rangle\\
\langle t(\omega),X(t(\omega))\rangle & \|X(t(\omega))\|^{2}+\|t(\omega)\|^{2}\\
\end{matrix}\right )$$
and
$$det\,h_{T}(\omega)=\|t(\omega)\|^{4}+\|t(\omega)\|^{2} \|X(t(\omega))\|^{2}-|\langle t(\omega),X(t(\omega))\rangle|^{2},\,w\in\mathbb{D}.$$
By using the Cauchy-Schwarz inequality, we obtain
$1\leq\frac{det\,h_{T}(\omega)}{K^{2}(\overline{\omega}, \overline{\omega})}\leq1+\|X\|^{2}$ for any $w\in\mathbb{D}$.
Note that $\mathcal{H}_{K}\in Lat(T)$ and $(T|_{\mathcal{H}_{K}})^{*}\sim_{s}M$.
\end{ex}

\begin{ex}
Let $M_{z}$ be the multiplication operator on Hardy space.
Let $T=\left(\begin{matrix}
T_{1} & S\\
0 & T_{1}\\
\end{matrix}\right)$, where $T_{1}$ is the adjoint of the multiplication operator on the Hilbert space $\mathcal{H}_1$, determined by kernel function $K_1(z,w)=\sum\limits_{n=0}^\infty\frac{2(n+1)}{n+2}z^n\overline{w}^n$ for $z,w\in\mathbb{D}$ and $S=diag(d_{0}, d_{1}, \cdots)$ respect to some orthonormal basis of $\mathcal{H}_1$.
Suppose that $\{d_{i}\}_{i=0}^{\infty}$ is a positive convergent sequence.
Choosing $t_1(w):=K_1(\cdot,\overline{w})$, we know that $t_{1}$ is a non-vanishing holomorphic section of $E_{T_{1}}$.
It is easy to see that there exists a holomorphic frame of Hermitian bundle $E_T$ in the form $\gamma_{1}=\left(\begin{matrix} t_{1} \\
0\\
\end{matrix}\right), \gamma_{2}=\left(\begin{matrix} t_{2} \\
t_{1}\\
\end{matrix}\right)$ for some holomorphic vector $t_2$.
By a routine computation, we obtain $t_{2}(\omega)=-\sum\limits_{i=1}^{\infty}\sqrt{\frac{2(i+1)}{i+2}}\Big(\sum\limits_{j=0}^{i-1}d_{j}\Big)z^{i}\omega^{i-1}.$
Note that $$\frac{\sqrt{2}}{2}\leq\frac{\|t_{1}(\omega)\|^{2}}{K(\overline{\omega},\overline{\omega})}=\frac{2(1-|\omega|^{2})\ln(1-|\omega|^{2})+2|\omega|^{2}}{|\omega|^{4}}\leq1$$ and
$$\|t_{2}(\omega)\|^{2}
=\sum\limits_{i=1}^{\infty}\frac{2(i+1)}{i+2}(\sum\limits_{j=0}^{i-1}d_{j})^{2}|\omega|^{2(i-1)}
\leq\Big(\sum\limits_{j=0}^{\infty}d_{j}\Big)^{2}\|t_{1}(\omega)\|^{2}$$
for all $\omega\in\mathbb{D}$.
By using Cauchy-Schwarz inequality, we have
$$\frac{1}{2}\leq\frac{\|t_{1}(\omega)\|^{4}}{K^{2}(\overline{\omega},\overline{\omega})}\leq\frac{det\,h_{T}(\omega)}{det\,h_{M_{z}^{*(2)}}(\omega)}\leq\frac{\|t_{1}(\omega)\|^{2}}{K(\overline{\omega},\overline{\omega})}\biggl(\frac{\|t_{1}(\omega)\|^{2}}{K(\overline{\omega},\overline{\omega})}+\frac{\|t_{2}(\omega)\|^{2}}{\|t_{1}(\omega)\|^{2}}\frac{\|t_{1}(\omega)\|^{2}}{K(\overline{\omega},\overline{\omega})}\biggl).$$
From inequation above and the convergence of the sequence $\{d_{i}\}_{i=0}^{\infty}$, we know $\frac{det\,h_{T}(\omega)}{det\,h_{M_{z}^{*(2)}}(\omega)}$ is bounded and bounded below from zero  for any $\omega\in \mathbb{D}$. On the other hand, since Theorem \ref{alger1}, we obtain $T^*|_{\mathcal{H}_1}$ is similar to $M_{z}$.
\end{ex}

\section{Reducibility of operators}

In this section, we will discuss the effect of the $n$-hypercontractivity of an operator on the structure of the operator, $n\geq1$. When the contractivity of a part of an operator reaches a certain degree, it can be inferred that the whole operator is reducible.

\subsection{Contractibility and reducibility of operators.}
In this subsection, we mainly study the properties of contractive operators.

\begin{lem}\label{yasuo1}
Let $T\in\mathcal{L}(\mathcal{H})$. Suppose that $T=(T_{ij})_{i, j=1}^{n}$ with respect to some decomposition $\mathcal{H}=\bigoplus\limits_{i=1}^{n}\mathcal{H}_{i}$.
If $T$ is contractive, then $T_{ij}$, $1\leq i,j\leq n$ are contractive operators.
\end{lem}
\begin{proof}
If $T$ is contractive, then for any $x\in \mathcal{H}$, we have $\|Tx\|\leq \|x\|$. Setting $$y_{j}=(\underbrace{0, \cdots, 0 }_{\text{j-1}}, x_{j}, \underbrace{0, \cdots, 0 }_{\text{n-j}})^{T}\in \bigoplus\limits_{i=1}^{n}\mathcal{H}_{i}. $$ It follows that
$$\|Ty_{j}\|^{2}=\sum\limits_{i=1}^{n}\|T_{ij}x_{j}\|^{2}\leq\|x_{j}\|^{2}.$$
By the property of norm, we obtain $$\sum\limits_{i=1}^{n}\|T_{ij}\|^{2}\leq1, \|T_{ij}\|\leq1.$$
Therefore, $T_{ij}$, $1\leq i,j\leq n$ are contractive operators.
\end{proof}
\begin{rem}\label{zhu9.9}
In Lemma \ref{yasuo1}, a similar method can be obtained $\sum\limits_{j=1}^{n}\|T_{ij}\|^{2}\leq1.$ If there exists $\|T_{i_{0}j_{0}}\|=1$, then $\|T_{i, j_{0}}\|=0$ and $\|T_{i_{0}, j}\|=0$, that is,
$T_{i, j_{0}}=0$ and $T_{i_{0}, j}=0$ for $i\neq i_{0}$ and $j\neq j_{0}$.
\end{rem}
\begin{prop}\label{yl9.82}
Let $T\in\mathcal{L}(\mathcal{H})$ be a contractive operator. Suppose that $T=(T_{ij})_{i, j=1}^{n}$ with respect to some decomposition $\mathcal{H}=\bigoplus\limits_{i=1}^{n}\mathcal{H}_{i}$, $n>1$.
If there exists $1\leq i_{0}\leq n$ such that $\|T_{i_{0}i_{0}}\|=1$, then $T$ is reducible.
\end{prop}
\begin{proof}
If there exists $1\leq i_{0}\leq n$ such that $\|T_{i_{0}i_{0}}\|=1$, by Remark \ref{zhu9.9}, we have
\begin{equation}\label{9.84}
T_{i_{0}j}=0,\ T_{ki_{0}}=0,\ j,k\neq i_{0}.
\end{equation}
We denote $\widetilde{T}$ by the operator $T$ satisfy equation (\ref{9.84}) and denote the identity operator on $\mathcal{H}$ by
$$I_{\mathcal{H}}=\left(\begin{matrix}I_{\mathcal{H}_{1}} & 0 & \cdots & 0 \\
0 & I_{\mathcal{H}_{2}} & \cdots & 0 \\
\vdots&\vdots & \ddots & \vdots\\
0 & 0 & \cdots &  I_{\mathcal{H}_{n}}\\
\end{matrix}\right), $$ where $I_{\mathcal{H}_{i}}$ is the identity operator on $\mathcal{H}_{i}$.
Let $I_{\mathcal{H}}(i, j)_{r}$ and $I_{\mathcal{H}}(i, j)_{c}$ represent operators that replaces $i$th and $j$th rows or columns of $I_{\mathcal{H}}$, respectively.
When $n=2$, by Remark \ref{zhu9.9}, it is easy to see that $T$ is reducible. For $n>2$, setting $U:=I_{\mathcal{H}}(1, i_{0})_{r}$. We have $U=U^{*}=U^{-1}=I_{\mathcal{H}}(1, i_{0})_{r}=I_{\mathcal{H}}(1, i_{0})_{c}$
and $$U\widetilde{T}U^{*}=T_{i_{0}i_{0}}\oplus(T_{ij})_{i, j\neq i_{0}}, $$where $(T_{ij})_{i, j\neq i_{0}}$ is obtained by removing $i_{0}$ row and $i_{0}$ column from $(T_{i, j})_{i, j=1}^{n}$.
Obviously, $U\widetilde{T}U^{*}$ is reducible. Recall that the unitary transformation does not change the reducibility of the operator, so the conclusion is valid.
\end{proof}
From the above discussion, it is intuitive to see that the reducibility of an operator $T$, when $T$ is contractive, can be characterized by a certain part.
Before proving the following lemma, we recall two concepts.
An operator $V$ in $\mathcal{L}(\mathcal{H})$ is an isometry if and only if $V^*V=I$.
A non negative definite kernel $K:\Omega\times\Omega\rightarrow \mathcal{M}_n(\mathbb{C})$ is said to be normalized at $w_0\in\Omega$ if
there exist a neighborhood $\Omega_0$ of $w_0\in\Omega$ such that $K(z,w_0)=I_{n\times n}$ for all $z\in\Omega$.

\begin{lem}\label{zylem}
Let $T\in B_{1}(\mathbb{D})$. Then $T^{*}$ is an isometric operator if and only if $T\sim_{u}(M_{z}^{*}, \mathcal{H}_{K}, K)$, where $\mathcal{H}_{K}$ is the analytic functional Hilbert space with reproducing
kernel $K(z, \omega)=\frac{1}{1-z\overline{\omega}}$, $z,w\in\mathbb{D}$.
\end{lem}
\begin{proof}
For $M_{z}^{*}$ on $\mathcal{H}_{K}$, we know that $M_{z}^{*}M_{z}=I$.
If there exists a unitary operator $U$, such that $T=UM_{z}^{*}U^{*}$, then $TT^{*}=UM_{z}^{*}U^{*}UM_{z}U^{*}=I$. So $T^{*}$ is isometric.

Conversely, it is well known that $T$ is unitarily equivalent to $M_z^*$ on some Hilbert space determined by the reproducing kernel $\widetilde{K}:\mathbb{D}\times\mathbb{D}\rightarrow \mathbb{C}$ and $\widetilde{K}(z,w)\neq0$ for all $z,w\in\mathbb{D}$.
Let $\widehat{K}(z,w)=\psi(z)\widetilde{K}(z,w)\psi(w)^*$, where $\psi(z)=\widetilde{K}(0,0)^{\frac{1}{2}}\widetilde{K}(z,0)^{-1}$ for each $z\in\mathbb{D}$.
It is easy to verify that $\widehat{K}$ is a non negative definite kernel over $\mathbb{D}$ and is normalized at 0.
Then there exist a unique Hilbert space $\mathcal{H}_{\widehat{K}}$ with reproducing kernel $\widehat{K}$.
By Remark 3.8 of \cite{CS2} given by R. E. Curto and N. Salinas, we have $M_z$ acting on $\mathcal{H}_{\widehat{K}}$ and $\mathcal{H}_{\widetilde{K}}$ are unitarily equivalent.
Thus, $M_z^*\in B_{1}(\mathbb{D})$ is unitarily equivalent to $T$ and $M_z$ is an isometric operator on $\mathcal{H}_{\widehat{K}}$.

Let $S$ be the adjoint of multiplication operator on $\mathcal{H}_{\widehat{K}}$ and $e(w)=\widehat{K}(\cdot,\overline{w})$ for every $w\in\mathbb{D}$.
Then $e$ is a non-zero holomorphic section of $E_S$ and $Se(\omega)=\omega e(\omega)$, $\omega\in\mathbb{D}$ on $\mathcal{H}_{\widehat{K}}$.
Further, since $S^{*}$ is isometric, we have $SS^{*}\omega e(\omega)=we(\omega)$ and $S(S^{*}\omega e(\omega)-e(\omega))=0$.
Then we will find a holomorphic function $\lambda$ on $\mathbb{D}$ satisfy $S^{*}\omega e(\omega)=e(\omega)+\lambda(\omega)e(0)$, $\omega\in\mathbb{D}$ (see details in Proposition 2.4 \cite{RS2}).
It follows from $S^{*}$ is isometric and $\widehat{K}$ is normalized at 0 that
$$\begin{array}{lll}
|\omega|^{2}\|e(\omega)\|^{2}
&=&\|e(\omega)+\lambda(\omega)e(0)\|^{2}\\
&=&\|e(\omega)\|^{2}+\lambda(\omega)\langle e(0), e(\omega)\rangle+\overline{\lambda(\omega)}\langle e(\omega), e(0)\rangle+|\lambda(\omega)|^{2}\|e(0)\|^{2}\\
&=&\|e(\omega)\|^{2}+\lambda(\omega)+\overline{\lambda(\omega)}+|\lambda(\omega)|^{2}.
\end{array}$$
Let $\phi(\omega)=1+\lambda(\omega)$ for $w\in\mathbb{D}$.
Therefore, we obtain
\begin{equation}\label{deng3}
\|e(\omega)\|^{2}(1-|\omega|^{2})=1-|\phi(\omega)|^{2}.
\end{equation}

Note that $\|S^{*}\|=\|S\|=1$.
By Theorem \ref{9.211}, we have $E_{S}\sim_uE_{M_{z}^{*}}\otimes \mathcal{E}$, where $M_{z}^{*}$ acting on $\mathcal{H}_{K}$ and $\mathcal{E}$ is a bundle with rank one.
Then there exists a holomorphic vector-valued function $r$ on $\mathbb{D}$ such that $$\|e(\omega)\|^{2}=\frac{1}{1-|\omega|^{2}}\|r(\omega)\|^{2}.$$
Combining with equation (\ref{deng3}), we have
$\|r(\omega)\|^{2}+|\phi(\omega)|^{2}=1$. That means $\|r'(\omega)\|^{2}+|\phi'(\omega)|^{2}=0$ and $\phi(\omega)=c\in \mathbb{C}$. Then
$\|e(\omega)\|^{2}(1-|\omega|^{2})=1-c.$ A simple calculation shows that $\mathcal{K}_{S}(\omega)=-\frac{1}{(1-|w|^2)^2}$ for $w\in\mathbb{D}$.
From Theorem \ref{CDT1}, we have $M_{z}^{*}$ on $\mathcal{H}_{K}$ is similar to $S$ and then to $T$.
\end{proof}
\begin{prop}\label{pro1}
Let $T\in \mathcal{L}(\mathcal{H})$ be a contractive operator. If there exists $\mathcal{M}\in Lat(T)$, such that $T|_{\mathcal{M}}\in B_{1}(\mathbb{D})$ and $(T|_{\mathcal{M}})^{*}$ is an isometric operator, then $T$ is reducible.
\end{prop}
\begin{proof}
Since $\mathcal{M}\in Lat(T)$, $T$ has the following form under the space decomposition on $\mathcal{H}=\mathcal{M}\oplus\mathcal{M}^{\perp}$:
$$T=\left ( \begin{matrix}T|_{\mathcal{M}} & P_{\mathcal{M}}T|_{\mathcal{M}^{\perp}}\\
0 & P_{\mathcal{M}^{\perp}}T|_{\mathcal{M}^{\perp}}\\
\end{matrix}\right ):=\left ( \begin{matrix}T_{1} & T_{12}\\
0 & T_{2}\\
\end{matrix}\right ).$$
If $T_{1}\in B_{1}(\mathbb{D})$, $T_{1}^{*}$ is an isometric operator, then by Lemma \ref{zylem}, there exists a unitary operator $U$, such that  $M_{z}^{*}=UT_{1}U^{*}$, where $M_{z}$ is the multiplication operator on the analytic functional Hilbert space $\mathcal{H}_K$ with reproducing kernel $K(z, \omega)=\frac{1}{1-z\overline{\omega}}$, $z,w\in\mathbb{D}$. Thus, we have
$$\begin{array}{lll}
(U\oplus I)T(U\oplus I)^{*}
&=&\left ( \begin{matrix}UT_{1}U^{*} & \ UT_{12}\\
0 & T_{2}\\
\end{matrix}\right)
=\left ( \begin{matrix}M_{z}^{*} & \ UT_{12}\\
0 & T_{2}\\
\end{matrix}\right)
:=\widetilde{T}.
\end{array}$$
We know that the unitary operator does not change the reducibility and contractility of the operator, so we only need to show that the reducibility of the contractive operator $\widetilde{T}$.
Note that
$$\widetilde{T}^{*}\widetilde{T}=\left(\begin{matrix}M_{z}M_{z}^{*} & M_{z}UT_{12}\\
T_{12}^{*}U^{*}M_{z}^{*} & \ T_{12}^{*}T_{12}+T_{2}^{*}T_{2}\\
\end{matrix}\right)$$
is contractive and $M_{z}^{*}M_{z}=I$, then for any $x\in \mathcal{H}_K$, we have
$$\widetilde{T}^{*}\widetilde{T}\left ( \begin{matrix}M_{z}x\\
0\\
\end{matrix}\right)=\left(\begin{matrix}M_{z}M_{z}^{*}M_{z}x\\
T_{12}^{*}U^{*}M_{z}^{*}M_{z}x\\
\end{matrix}\right )=\left(\begin{matrix}M_{z}x\\
T_{12}^{*}U^{*}x\\
\end{matrix}\right), $$
and
$$\biggl\|\left(\begin{matrix}M_{z}x\\
T_{12}^{*}U^{*}x\\
\end{matrix}\right)\biggl\|\leq \biggl\|\left ( \begin{matrix}M_{z}x\\
0\\
\end{matrix}\right )\biggl\|. $$
That is, $T_{12}^{*}U^{*}x=0, x\in\mathcal{H}_K$. By the arbitrariness of $x$, we obtain $UT_{12}=0$ and $T$ is reducible.
\end{proof}
A necessary condition for an operator to be contractive is given in Lemma \ref{yasuo1}. In the following we will discuss sufficient conditions.
Before that, a symbol that will be used below will be introduced.
For $T\in\mathcal{L}(\mathcal{H})$, we write $T\geq0$ or $0\leq T$ to denote that $T$ is positive, i.e. $\langle Tx,x\rangle\geq0$ for any $x\in\mathcal{H}$.
For two operators $T,\widetilde{T}\in\mathcal{L}(\mathcal{H})$, we write $T\geq \widetilde{T}$ if $T-\widetilde{T}\geq0$.
Suppose that $0\leq T\leq I$, then there exist $S\in\mathcal{L}(\mathcal{H})$ such that $T=S^*S$ and
$\langle Tx,x\rangle=\langle S^{*}Sx,x\rangle=\langle Sx,Sx\rangle=\|Sx\|^{2}\leq\|x\|^{2}$ for any $x\in\mathcal{H}$.
It follows that $\|T\|=\|S^{*}S\|=\|S\|^{2}\leq1$. So $T$ is contractive.
Therefore, we have the following proposition.
\begin{prop}\label{lemma1}
Let $T\in\mathcal{L}(\mathcal{H})$. If $T$ is positive and $T\leq I$, then $T$ is contractive.
\end{prop}

\begin{prop}
Let $T\in\mathcal{L}(\mathcal{H})$. If there exists $\mathcal{M}\in Lat(T)$, such that $$\|T|_{\mathcal{M}}\|^{2}\leq\frac{1}{2}, \|P_{\mathcal{M}}T|_{\mathcal{M}^{\bot}}\|^{2}\leq\frac{1}{2}\big(1-\|(T^{*}|_{\mathcal{M}^{\bot}})^{*}\|^{2}\big),$$ then $T$ is contractive.
\end{prop}
\begin{proof}
Since $\mathcal{M}\in Lat(T)$, $T$ has the following form under the space decomposition of $\mathcal{M}\oplus\mathcal{M}^{\perp}$:
$$T=\left(\begin{matrix}
T|_{\mathcal{M}} & P_{\mathcal{M}}T|_{\mathcal{M}^{\perp}}\\
0 & P_{\mathcal{M}^{\perp}}T|_{\mathcal{M}^{\perp}} \\
\end{matrix}\right)=\left(\begin{matrix}
T|_{\mathcal{M}} & P_{\mathcal{M}}T|_{\mathcal{M}^{\perp}}\\
0 & (T^{*}|_{\mathcal{M}^{\bot}})^{*} \\
\end{matrix}\right):=\left(\begin{matrix}
T_{1} & T_{12}\\
0 & T_{2}\\
\end{matrix}\right).$$
For any $x\in\mathcal{M}$ and $y\in\mathcal{M}^{\perp}$, we have
$$\begin{array}{lll}
\|T(x\oplus y)\|^{2}
&=&\|T_{1}x+ T_{12}y\|^{2}+\|T_{2}y\|^{2}\\
&\leq&(\|T_{1}x\|+\|T_{12}y\|)^{2}+\|T_{2}y\|^{2}\\
&\leq&2\|T_{1}x\|^{2}+2\|T_{12}y\|^{2}+\|T_{2}y\|^{2}\\
&\leq& 2\|T_{1}\|^{2}\|x\|^{2}+2\|T_{12}\|^{2}\|y\|^{2}+\|T_{2}\|^{2}\|y\|^{2}\\
&=&2\|T_{1}\|^{2}\|x\|^{2}+(2\|T_{12}\|^{2}+\|T_{2}\|^{2})\|y\|^{2}.
\end{array}$$
It follows from $\|T_{1}\|^{2}\leq\frac{1}{2}$ and $\|T_{12}\|^{2}\leq\frac{1}{2}(1-\|T_{2}\|^{2})$ that
$$2\|T_{1}\|^{2}\|x\|^{2}+(2\|T_{12}\|^{2}+\|T_{2}\|^{2})\|y\|^{2}\leq\|x\|^{2}+\|y\|^{2}.$$So we obtain
$\|T(x\oplus y)\|^{2}\leq\|x\|^{2}+\|y\|^{2}. $
By the arbitrariness of $x, y$, we know $T$ is a contractive operator.
\end{proof}

\begin{prop}
Let $T\in\mathcal{L}(\mathcal{H})$. If there exists $\mathcal{M}\in Lat(T)$ such that $1-(T|_{\mathcal{M}})^{*}T|_{\mathcal{M}}$ is positive and invertible and
\begin{equation}\label{20230401}
I-T^{*}|_{\mathcal{M}^{\bot}}(T^{*}|_{\mathcal{M}^{\bot}})^{*}\geq (P_{\mathcal{M}}T|_{\mathcal{M}^{\bot}})^{*}\big(I-T|_{\mathcal{M}}(T|_{\mathcal{M}})^{*}\big)^{-1}P_{\mathcal{M}}T|_{\mathcal{M}^{\bot}},
\end{equation}
then $T$ is contractive.
\end{prop}
\begin{proof}
Since $1-T|_{\mathcal{M}}^{*}T|_{\mathcal{M}}$ is positive, let $A$ be the square root of $1-(T|_{\mathcal{M}})^{*}T|_{\mathcal{M}}$.
From its invertibility, we can define the operator $B=-A^{-1}(T|_{\mathcal{M}})^{*}P_{\mathcal{M}}T|_{\mathcal{M}^{\bot}}$.
By equation (\ref{20230401}), there exists an operator $C$ such that
$C^*C=I-T^{*}|_{\mathcal{M}^{\bot}}(T^{*}|_{\mathcal{M}^{\bot}})^{*}- (P_{\mathcal{M}}T|_{\mathcal{M}^{\bot}})^{*}\big(I-T|_{\mathcal{M}}(T|_{\mathcal{M}})^{*}\big)^{-1}P_{\mathcal{M}}T|_{\mathcal{M}^{\bot}}$.
Note that $\big(I-T|_{\mathcal{M}}(T|_{\mathcal{M}})^{*}\big)^{-1}=I+T|_{\mathcal{M}}\big(I-(T|_{\mathcal{M}})^{*}T|_{\mathcal{M}}\big)^{-1}(T|_{\mathcal{M}})^{*}$, we have
$$\begin{array}{lll}
I-T^*T&=&\left(\begin{matrix} I-(T|_{\mathcal{M}})^{*}T|_{\mathcal{M}} & -(T|_{\mathcal{M}})^{*}P_{\mathcal{M}}T|_{\mathcal{M}^{\bot}}\\
-(P_{\mathcal{M}}T|_{\mathcal{M}^{\bot}})^*T|_{\mathcal{M}} &\ I-(P_{\mathcal{M}}T|_{\mathcal{M}^{\bot}})^*P_{\mathcal{M}}T|_{\mathcal{M}^{\bot}}- T^{*}|_{\mathcal{M}^{\bot}}(T^{*}|_{\mathcal{M}^{\bot}})^{*}\\
\end{matrix}\right)\\
&=&\left(\begin{matrix} A^2 & AB\\
B^*A & B^*B+C^*C \\
\end{matrix}\right )\\
&=&\left(\begin{matrix} A & 0\\
B^* & C^* \\
\end{matrix}\right )\left(\begin{matrix} A & B\\
0 & C \\
\end{matrix}\right )\\
&\geq&0.
\end{array}$$
Hence, $T$ is contractive.
\end{proof}
As a consequence, we have the following example.
\begin{ex}\label{li9.8}
Let $T_1$, $T_2$ and $T_{12}$ be operators on $\mathcal{H}$ satisfying $T_1e_i=a_{i-1}e_{i-1}$, $T_2e_i=b_{i-1}e_{i-1}$ and $T_{12}e_i=d_ie_i$ for some orthonormal basis $\{e_i\}_{i=1}^\infty$ of $\mathcal{H}$ and sequences $\{a_{i}\}_{i=1}^{\infty},\ \{b_{i}\}_{i=1}^{\infty}$ and $\{d_{i}\}_{i=1}^{\infty}$.
Suppose that $0\leq a_{i},b_{i},d_i\leq1$ for $i\geq 1$ and $\{a_{i}\}_{i=1}^{\infty}$ is always not 1.
If $d_{1}^{2}\leq1-a_{1}^{2}$ and $d_{i}^{2}\leq(1-a_{i}^{2})(1-b_{i-1}^{2})$ for $i\geq2$, then $T=\left(\begin{matrix} T_{1} & T_{12}\\
0 & T_{2} \\
\end{matrix}\right )$ is contractive.
\end{ex}
\begin{proof}
Since $0\leq a_{i}<1$ for $i\geq 1$, there exist operator $A$ defined by $Ae_1=1$ and $Ae_i=\sqrt{1-a_{1}^{2}}e_i$, $i>1$ such that $1-T_{1}^{*}T_{1}=A^2$.
Let $B=-A^{-1}T_1^*T_{12}$. By our hypotheses, we have $c_1:=1-\frac{d_1^2}{1-a_1^2}\geq0$ and $c_i:=1-b_{i-1}^2-\frac{d_i^2}{1-a_i^2}\geq0$ for $i>1$.
Define the operator $C$ as $Ce_i=c_ie_i$ for every $i\geq1$. Then $C$ is a positive operator.
Note that $\left(\begin{matrix} 1-T_{1}^{*}T_{1} & -T_{1}^{*}T_{1,2}\\
-T_{1,2}^{*}T_{1} & 1-T_{1,2}^{*}T_{1,2}-T_{2}^{*}T_{2} \\
\end{matrix}\right)=\left(\begin{matrix} A^2 & AB\\
B^*A &\ B^{*}B+C \\
\end{matrix}\right)$.
It follows that $1-T^{*}T=\left(\begin{matrix} A & 0\\
C^{*} & C^{\frac{1}{2}} \\
\end{matrix}\right)\left(\begin{matrix} A & B\\
0 & C^{\frac{1}{2}}\\
\end{matrix}\right )\geq0$. Thus, $T$ is a contraction.
\end{proof}

\subsection{N-hypercontractivity and reducibility of operators.}
It is shown in Lemma \ref{yasuo1} that when an operator $T$ is contractive, so are every part of $T$.
The following two discussions illustrate that this property cannot be fully generalized to $M$-contraction, even under the assumption of $n$-hypercontraction for $n>1$.
For a fixed Cowen-Douglas atom, we assume that its corresponding sequence of polynomials in (2) of Definition \ref{1.27} is $\{p_{l}(z, \overline{\omega})\}$,
where $p_{l}(z, \overline{\omega})=\sum\limits_{i,j=0}^{k_l}a_{ij}^lz^i\overline{\omega}^j$.
\begin{lem}\label{202302191}
Let $T\in B_{n}(\mathbb{D})\bigcap\mathcal{L}(\mathcal{H})$. Let $M$ be a Cowen-Douglas atom on some analytic functional Hilbert space $\mathcal{H}_{K}$ with reproducing kernel $K$.
Suppose that there exists $\mathcal{M}\in Lat(T)$.
If $T$ is $M$-contractive and $T|_{\mathcal{M}}\in B_{m}(\mathbb{D})$ for $m<n$, then $T|_{\mathcal{M}}$ is $M$-contractive.
\end{lem}
\begin{proof}

From the proof of Theorem \ref{xiangsi}, we have $T=\left(\begin{matrix}
T|_{\mathcal{M}} & P_{\mathcal{M}}T|_{\mathcal{M}^{\perp}}\\
0 & (T^{*}|_{\mathcal{M}^{\perp}})^{*} \\
\end{matrix}\right)$.
Note that for any positive integer $j$, $T^j=\left(\begin{matrix}
T_0^j & \triangle_j\\
0 & T_1^j \\
\end{matrix}\right)$, then we have $$T^{*i}T^j=\left(\begin{matrix}
T_0^{*i} & 0\\
\triangle_i^* & T_1^{*i} \\
\end{matrix}\right)\left(\begin{matrix}
T_0^j & \triangle_j\\
0 & T_1^j \\
\end{matrix}\right)=\left(\begin{matrix}
T_0^{*i}T_0^j & T_0^{*i}\triangle_j\\
\triangle_i^*T_0^j & \triangle_i^*\triangle_j+T_1^{*i}T_1^j \\
\end{matrix}\right)$$
and $$p_l(T^*,T)=\sum\limits_{i,j=0}^{k_l}a_{ij}^lT^{*i}T^j=\left(\begin{matrix}
p_l(T_0^*,T_0) & \sum\limits_{i,j=0}^{k_l}a_{ij}^lT_0^{*i}\triangle_j\\
\sum\limits_{i,j=0}^{k_l}a_{ij}^l\triangle_i^*T_0^j & \sum\limits_{i,j=0}^{k_l}a_{ij}^l\triangle_i^*\triangle_j+p_l(T_1^*,T_1) \\
\end{matrix}\right).$$
Since $T$ is $M$-contractive, $\sup\limits_l\|p_l(T^*,T)\|<\infty$. It follows that $\sup\limits_l\|p_l(T_0^*,T_0)\|<\infty$.

Recall that $C=\left(\begin{matrix}
C_0 & C_{01}\\
C_{10} & C_1 \\
\end{matrix}\right)\begin{matrix}
\mathcal{M} \\
\mathcal{M}^{\perp}\\
\end{matrix}\in\mathcal{L}(\mathcal{H})$ is a positive operator, then for any $x=x_0\oplus x_1\in\mathcal{M}\bigoplus\mathcal{M}^{\perp}$, $\langle Cx,x\rangle_{\mathcal{H}}\geq0$.
If $x_1=0$, then $\langle C_0x_0,x_0\rangle_{\mathcal{M}}\geq0$. From the arbitrariness of $x_0\in\mathcal{M}$, it can be obtained that $C_0\in\mathcal{L}(\mathcal{M})$ is positive.
Without loss of generality, we assume that $(WOT)-\lim\limits_{l\rightarrow\infty}p_l(T^*,T)=C$.
That means $\langle p_l(T^*,T)x-Cx,y\rangle\rightarrow0$ for each $x,y\in\mathcal{H}$. Let $x, y$ satisfy $P_{\mathcal{M}^{\perp}}x=P_{\mathcal{M}^{\perp}}y=0$.
Thus, $(WOT)-\lim\limits_{l\rightarrow\infty}P_{\mathcal{M}}p_l(T^*,T)P_{\mathcal{M}}=C_0$ and $T|_{\mathcal{M}}$ is $M$-contractive.
\end{proof}
The following proposition is an immediate consequence.
\begin{prop}\label{9.132}
Let $T\in\mathcal{L}(\mathcal{H})$. If $T$ is $n-$hypercontractive, then for any $\mathcal{M}\in Lat(T)$, $T|_{\mathcal{M}}$ is also.
\end{prop}

More recently, the first and fourth authors, along with H. Kwon in \cite{JKX} proved that a necessary condition for a weighted backward shift operator $T$ with a weight sequence of $\{w_i\}_{i=1}^\infty$ to be an $n$-hypercontraction is that $w_i\leq\frac{i}{n+i-1}$.
By using this result, we show that in Proposition \ref{9.132} $(T^{*}|_{\mathcal{M}^{\perp}})^{*}$ is not necessarily $2$-hypercontractive when $n=2$.

\begin{ex}
Let $T=\left(\begin{matrix} T_{1} & T_{12}\\
0 & T_{2} \\
\end{matrix}\right)\in\mathcal{L}(\mathcal{H}\oplus\mathcal{H})$. Then there exists a not $2-$hypercontractive operator $T_{2}$, such that $T$ is $2-$hypercontractive.
\end{ex}
\begin{proof}
If $T$ is $2-$hypercontractive, then $T_{1}, T_{12}, T_{2}$ are contractive by Lemma \ref{yasuo1} and $$I-2T^{*}T+T^{*2}T^{2}=\left(\begin{matrix} I-2T_{1}^{*}T_{1}+T_{1}^{*2}T_{1}^{2} & -2T_{1}^{*}T_{12}+T_{1}^{*2}(T_{1}T_{12}+T_{12}T_{2})\\
-2T_{12}^{*}T_{1}+(T_{12}^{*}T_{1}^{*}+T_{2}^{*}T_{12}^{*})T_{1}^{2} & \Delta\\
\end{matrix}\right)$$
is positive, where $\Delta=I-2T_{2}^{*}T_{2}+T_{2}^{*2}T_{2}^{2}-2T_{12}^{*}T_{12}+(T_{12}^{*}T_{1}^{*}+T_{2}^{*}T_{12}^{*})(T_{1}T_{12}+T_{12}T_{2})$.
Assume that $T_{1}$ and $T_{2}$ are weighted backward shift operators with weight sequences $\big\{\sqrt{\frac{i}{i+3}},i\geq1\big\}$ and $\big\{\sqrt{\frac{1}{2}}+x,\sqrt{\frac{i}{i+1}},i\geq2\big\}$ for $x>0$ and $T_{12}$ satisfies $T_{12}e_{i}=
\begin{cases}
ye_{1}, i=1\\
0, i> 1
\end{cases}$for $y\neq0$ respect to an orthonormal basis $\{e_i\}_{i=1}$ of $\mathcal{H}$, respectively.
We know $T_{2}$ is not a 2-hypercontraction by Theorem 1.1 of \cite{JKX}.
Next, we will prove that there exist suitable $x$ and $y$ such that $T$ is 2-hypercontractive.

Define operators $A$ and $B$ on $\mathcal{H}$ as $Ae_i=\sqrt{\frac{6}{(i+1)(i+2)}}e_i$ for $i\geq1$ and $Be_{i}=
\begin{cases}
-\sqrt{2}ye_{2}, i=1\\
(\sqrt{\frac{1}{6}}+\sqrt{\frac{1}{3}}x)ye_{3}, i=2.\\
0, i>2
\end{cases}$
Note that the solution of
$$\begin{cases}
1-4y^2\geq0\\
1+(\sqrt{\frac{1}{2}}+x)^2(\frac{2}{3}y^2-2)\geq0
\end{cases}$$ is not an empty set for $x>0$. Then $C:=diag\Big(1-4y^2,1+(\sqrt{\frac{1}{2}}+x)^2(\frac{2}{3}y^2-2),-\frac{1}{3}+\frac{2}{3}(\sqrt{\frac{1}{2}}+x)^2,0,\cdots\Big)$ respect to the orthonormal basis $\{e_i\}_{i=1}$ of $\mathcal{H}$.
A routine verification shows that
$$I-2T^{*}T+T^{*2}T^{2}=\left(\begin{matrix} A^{2} & AB\\
B^{*}A &\ B^{*}B+C \\
\end{matrix}\right)=\left(\begin{matrix} A & 0\\
B^{*} & C^{\frac{1}{2}} \\
\end{matrix}\right)\left(\begin{matrix} A & B\\
0 & C^{\frac{1}{2}} \\
\end{matrix}\right ). $$
Thus, we obtain $T$ is 2-hypercontractive. This completes the proof.
\end{proof}
Before discussing the relationship between the $n-$hypercontractivity and reducibility of operators, we need the following lemma.

\begin{lem}\label{lemma4}
If $T\in\mathcal{L}(\mathcal{H})$ is $n-$hypercontractive, then operator $\sum\limits_{j=1}^{n}(-1)^{j+1}{n \choose j}(T^{*})^{j}T^{j}$ is contractive.
\end{lem}
\begin{proof}
If $T\in\mathcal{L}(\mathcal{H})$ is $n-$hypercontractive, we have
\begin{equation}\label{budengs1}
\sum\limits_{j=0}^{k}(-1)^{j}{k \choose j}(T^{*})^{j}T^{j}
\end{equation}
is positive for any $1\leq k\leq n$. Adding up these equations, we arrive at
\begin{equation}\label{202304021}
\sum\limits_{j=1}^{n}(-1)^{j-1}\biggl(\sum\limits_{i=j-1}^{n-1}{i \choose j-1}\biggl)(T^{*})^{j-1}T^{j-1}\geq0.
\end{equation}
It is well known that ${l \choose k}={l-1 \choose k-1}+{l-1 \choose k}$ for positive integers $1\leq k\leq l-1$.
By using this equation repeatedly, we have
$$\begin{array}{lll}
{n \choose j}
&=&{n-1 \choose j-1}+{n-1 \choose j}\\
&=&{n-1 \choose j-1}+{n-2 \choose j-1}+{n-2 \choose j}\\
&=&{n-1 \choose j-1}+{n-2 \choose j-1}+{n-3 \choose j-1}+{n-3 \choose j}\\
&=&{n-1 \choose j-1}+{n-2 \choose j-1}+\cdots+{j+1 \choose j-1}+{j \choose j-1}+{j \choose j}\\
&=&{n-1 \choose j-1}+{n-2 \choose j-1}+\cdots+{j+1 \choose j-1}+{j \choose j-1}+{j-1 \choose j-1}
\end{array}$$
for every $1\leq j\leq n-1$. Further, we infer that inequality (\ref{202304021}) is equivalent to the following
$$\sum\limits_{j=1}^{n}(-1)^{j-1}{n \choose j}(T^{*})^{j-1}T^{j-1}\geq0.$$
Recall that for any positive operator $A$ and bounded operator $B$, $B^*AB$ is also positive.
The above formula can be converted into
\begin{equation}\label{dengs3}
T^{*}\biggl(\sum\limits_{j=1}^{n}(-1)^{j-1}{n \choose j}(T^{*})^{j-1}T^{j-1}\biggl)T=-\sum\limits_{j=1}^{n}(-1)^{j}{n \choose j}(T^{*})^{j}T^{j}\geq0.
\end{equation}
Note that inequality (\ref{budengs1}) is $-\sum\limits_{j=1}^{n}(-1)^{j}{n \choose j}(T^{*})^{j}T^{j}\leq I$ when $j=n$.
Combining with inequation (\ref{dengs3}), we have
$$0\leq-\sum\limits_{j=1}^{n}(-1)^{j}{n \choose j}(T^{*})^{j}T^{j}=\sum\limits_{j=1}^{n}(-1)^{j+1}{n \choose j}(T^{*})^{j}T^{j}\leq1.$$
From Proposition \ref{lemma1}, we obtain that $\sum\limits_{j=1}^{n}(-1)^{j+1}{n \choose j}(T^{*})^{j}T^{j}$ is contractive.
\end{proof}
We will give a sufficient condition for a operator to be reducible by using the $n$-hypercontractivity and the curvature.
\begin{thm}\label{pro2}
Let $T\in \mathcal{L}(\mathcal{H})$ be an $n-$hypercontraction, $n\geq1$. If there exists $\mathcal{M}\in Lat(T)$ such that $T|_{\mathcal{M}}\in B_{1}(\mathbb{D})$ and
$\mathcal{K}_{T|_{\mathcal{M}}}(\omega)=\frac{-n}{(1-|\omega|^{2})^{2}}$ for every $\omega\in\mathbb{D}$, then $T$ is reducible.
\end{thm}
\begin{proof}

If $T|_{\mathcal{M}}\in B_{1}(\mathbb{D})$ and $\mathcal{K}_{T|_{\mathcal{M}}}(\omega)=\frac{-n}{(1-|\omega|^{2})^{2}}$ for $\omega\in\mathbb{D}$,
by Theorem \ref{CDT1} given by M. J. Cowen and R. G. Douglas, the curvature is the completely unitary invariant of Cowen-Douglas operator with index one and by the main result of \cite{Misra}, we know that
$T|_{\mathcal{M}}$ is unitarily
equivalent to $(M_{z}^{*}, \mathcal{H}_{K}, K)$, where $\mathcal{H}_{K}$ is the analytic functional Hilbert space with reproducing kernel $K(z,w)=\frac{1}{(1-z\overline{w})^n}$ over  $\mathbb{D}\times\mathbb{D}$.
It is well known that the multiplication operator on Hardy space is an isometric operator. So by Proposition \ref{pro1}, we obtain that this proposition holds when $n=1$.
We will assume that $n>1$.

There exists a unitary operator $U:\mathcal{M}\rightarrow\mathcal{H}_{K}$ such that $M_{z}^{*}=UT|_{\mathcal{M}}U^{*}$. Hence, we have
$$\begin{array}{lll}
(U\oplus I)T(U\oplus I)^{*}
&=&\left ( \begin{matrix}M_{z}^{*}\ & UP|_{\mathcal{M}}T|_{\mathcal{M}^{\perp}}\\
0 & P|_{\mathcal{M}^{\perp}}T|_{\mathcal{M}^{\perp}}\\
\end{matrix}\right )
:=\left ( \begin{matrix}M_{z}^{*} & T_{12}\\
0 & T_{2}\\
\end{matrix}\right )
:=\widetilde{T}.
\end{array}$$
Since $U\oplus I$ is a unitary operator and the
unitary transformation does not change the reducibility and $n-$hypercontractivity of operators, we only need to verify that the $n-$hypercontractive operator $\widetilde{T}$ is reducible.
Space $\mathcal{H}_K$ has an orthonormal basis of the form $\{e_{i}(z)=\frac{(n)_i}{i!}z^i\}_{n=0}^{\infty}$, where $(n)_i$ is the Pochhammer symbol given by $\frac{\Gamma(n+i)}{\Gamma(n)}$.
Thus, $M_{z}^{*}$ can be regarded as a backward weighted shift operator respect to the orthonormal basis $\{e_{i}\}_{n=0}^{\infty}$ of $\mathcal{H}_K$
with weight sequence $\{\sqrt{\frac{i+1}{n+i}}\}_{i=0}^{\infty}$ up to unitary equivalence.
A routine calculation shows that $\sum\limits_{j=0}^{n}(-1)^{j}{n \choose j}M_z^{j}M_z^{*j}$ is a diagonal operator respect to $\{e_{i}\}_{n=0}^{\infty}$ and the first entry on the diagonal is 1 and the rest are $c_j$, $j\geq1$, where
$$c_j=\begin{cases}
1+\sum\limits_{i=1}^j(-1)^i{n \choose i}\frac{(j+1-i)_i}{(n+j-i)_i},\ 1\leq j\leq n;\\
1+\sum\limits_{i=1}^n(-1)^i{n \choose i}\frac{(j+1-i)_i}{(n+j-i)_i},\ j\geq n.
\end{cases}$$
Since $(1-x)^{n}=\sum\limits_{i=0}^{n}{n \choose i}(-x)^{i}$ and $\frac{1}{(1-x)^{n}}=\sum\limits_{i=0}^{\infty}{n-1+i \choose i}x^{i}$ for $|x|<1$, we have that
$$\begin{array}{lll}
1&=&(1-x)^{n}\frac{1}{(1-x)^{n}}\\
&=&\Big[\sum\limits_{j=0}^{n}{n \choose j}(-x)^{j}\Big]\Big[\sum\limits_{i=0}^{\infty}{n-1+i \choose i}x^{i}\Big]\\
&=&\sum\limits_{j=0}^{n}\Big[\sum\limits_{i=0}^{j}(-1)^{i}{n \choose i}{n-i+j-1 \choose j-i}\Big]x^{j}+\sum\limits_{j=n+1}^{\infty}\Big[\sum\limits_{i=0}^{n}(-1)^{i}{n \choose i}{n-i+j-1 \choose j-i}\Big]x^{j}.
\end{array}$$
It follows that $\sum\limits_{i=0}^{j}(-1)^{i}{n \choose i}{n-i+j-1 \choose j-i}=0$, $0\leq j\leq n$
and $\sum\limits_{i=0}^{n}(-1)^{i}{n \choose i}{n-i+j-1 \choose j-i}=0$, $j>n$. Further, we obtain that $c_j=0$ for any $j\geq1$ and $\sum\limits_{j=0}^{n}(-1)^{j}{n \choose j}M_z^{j}M_z^{*j}=I|_{span\{e_0\}}$, where $I$ is identity of $(U\oplus I)\mathcal{H}$.
Thus, we have
\begin{equation}\label{dengs5}
-\sum\limits_{j=1}^{n}(-1)^{j}{n \choose j}M_{z}^{j}(M_{z}^{*})^{j}e_{i}=\sum\limits_{j=1}^{n}(-1)^{j+1}{n \choose j}M_{z}^{j}(M_{z}^{*})^{j}e_{i}=e_{i},\ i\geq1.
\end{equation}
Note that
$$\begin{array}{lll}
& &\sum\limits_{j=1}^{n}(-1)^{j+1}{n \choose j}(\widetilde{T}^{*})^{j}\widetilde{T}^{j}\\
&=&\left(\begin{smallmatrix}\sum\limits_{j=1}^{n}(-1)^{j+1}{n \choose j}M_{z}^{j}(M_{z}^{*})^{j} & {n \choose 1}M_zT_{12}+\sum\limits_{j=2}^{n}(-1)^{j+1}{n \choose j}M_{z}^{j}\Big(\sum\limits_{l=0}^{j-1}(M_{z}^{*})^{l}T_{12}T_{2}^{j-l-1}\Big)\\
{n \choose 1}T_{12}^*M_{z}^{*}+\sum\limits_{j=2}^{n}(-1)^{j+1}{n \choose j}\Big(\sum\limits_{l=0}^{j-1}(T_{2}^{*})^{l}T_{12}^{*}M_{z}^{j-l-1}\Big)(M_{z}^{*})^{j}
 &\triangle\\
\end{smallmatrix}\right),
\end{array}$$
where $\triangle=\sum\limits_{j=1}^{n}(-1)^{j+1}{n \choose j}\biggl[(T_{2}^{*})^{j}T_{2}^{j}+\Big(\sum\limits_{h=0}^{j-1}(T_{2}^{*})^{h}T_{1,2}^{*}M_{z}^{j-h-1}\Big)\Big(\sum\limits_{h=0}^{j-1}(M_{z}^{*})^{h}T_{1,2}T_{2}^{j-h-1}\Big)\biggl]$.
Since equation (\ref{dengs5}), we have
$$\sum\limits_{j=1}^{n}(-1)^{j+1}{n \choose j}(\widetilde{T}^{*})^{j}\widetilde{T}^{j}\left ( \begin{matrix}e_{1}\\
0\\
\end{matrix}\right )=\left ( \begin{matrix}e_{1}\\
\sqrt{n}T_{1,2}^{*}e_{0}\\
\end{matrix}\right ).$$
By Lemma \ref{lemma4}, we know operator $\sum\limits_{j=1}^{n}(-1)^{j+1}{n \choose j}(\widetilde{T}^{*})^{j}\widetilde{T}^{j}$ is contractive. Then $T_{1,2}^{*}e_{0}=0$. 
To complete the proof by induction, suppose that $T_{1,2}^{*}e_{i}=0$ for $1\leq i\leq m-2$ for $m\geq2$. We will prove that $T_{1,2}^{*}e_{m-1}=0$. By the induction hypothesis, we have
$$\begin{array}{lll}
\sum\limits_{j=2}^{n}(-1)^{j+1}{n \choose j}\Big(\sum\limits_{l=0}^{j-1}(T_{2}^{*})^{l}T_{12}^{*}M_{z}^{j-l-1}\Big)(M_{z}^{*})^{j}e_m=\begin{cases}
\sum\limits_{j=2}^{m}(-1)^{j+1}{n \choose j}\prod\limits_{i=2}^jw_{m-i}^2w_{m-1}T_{12}^*e_{m-1},\ m<n;\\
\sum\limits_{j=2}^{n}(-1)^{j+1}{n \choose j}\prod\limits_{i=2}^jw_{m-i}^2w_{m-1}T_{12}^*e_{m-1},\ m\geq n,
\end{cases}
\end{array}$$
where $w_k=\sqrt{\frac{k+1}{n+k}}$ for $k\geq0$. It follows from Lemma \ref{lemma4} again that
$$\begin{cases}
\Big[{n \choose 1}+\sum\limits_{j=2}^{m}(-1)^{j+1}{n \choose j}\prod\limits_{i=2}^jw_{m-i}^2\Big]w_{m-1}T_{12}^*e_{m-1}=0,\ m<n;\\
\Big[{n \choose 1}+\sum\limits_{j=2}^{n}(-1)^{j+1}{n \choose j}\prod\limits_{i=2}^jw_{m-i}^2\Big]w_{m-1}T_{12}^*e_{m-1}=0,\ m\geq n.
\end{cases}$$
Since the coefficients of $T_{12}^*e_{m-1}$ in both cases are $\frac{1}{w_{m-1}}\neq0$. That means $T_{12}^*e_{m-1}=0$. This completes the induction step.
Then we obtain $T_{12}^{*}e_{i}=0$ for orthonormal basis $\{e_{i}\}_{i=0}^{\infty}$. Hence, $T_{12}^{*}=0$ and $\widetilde{T}$ is reducible.
\end{proof}
\begin{cor}
Let $T\in \mathcal{L}(\mathcal{H})$ is irreducible and $\|T\|\leq1$. If there exist $\mathcal{M}\in Lat(T)$ and a positive integer $k$ such that $T|_{\mathcal{M}}\in B_{1}(\mathbb{D})$ and
$\mathcal{K}_{T|_{\mathcal{M}}}(\omega)=\frac{-k}{(1-|\omega|^{2})^{2}}$ for all $\omega\in\mathbb{D}$, then $T$ cannot be subnormal.
\end{cor}
\begin{proof}
Suppose that the operator $T$ is subnormal. We know from \cite{Agler2} that an operator is $n-$hypercontractive for all positive integer $n$ if and only if it is a subnormal contraction.
Based on hypothesis $\|T\|\leq1$, we have that $T$ is $n-$hypercontractive for all positive integer $n$.
Since $T|_{\mathcal{M}}\in B_{1}(\mathbb{D})$ and $\mathcal{K}_{T|_{\mathcal{M}}}(\omega)=\frac{-k}{(1-|\omega|^{2})^{2}}$ for all $\omega\in\mathbb{D}$, by Theorem \ref{pro2},
we obtain $T$ is reducible. It's a contradiction. Hence, $T$ is not subnormal.
\end{proof}
\begin{cor}\label{9.6}
Let $T\in \mathcal{L}(\mathcal{H})$ be an $n-$hypercontraction. If there exists $\mathcal{M}\in Lat(T)$ such that $T|_{\mathcal{M}}\in B_{1}(\mathbb{D})$ and
\begin{equation}\label{dengs6}
\sum\limits_{i=0}^{n}(-1)^{i}{n \choose i}(T|_{\mathcal{M}})^{*i}(T|_{\mathcal{M}})^{i}=e\otimes e
\end{equation}
for some $e\in \mathcal{M}$ and $\|e\|=1$, then $T$ is reducible.
\end{cor}
\begin{proof}
Since $e\in \mathcal{M}$ and $\|e\|=1$, $e$ can be extended to an orthonormal basis $\{e_{n}\}_{n=0}^{\infty}$ of $\mathcal{M}$, where $e_0=e$.
We first prove the fact that for operator $T|_{\mathcal{M}}\in B_{1}(\mathbb{D})$, there exists an open set $\Omega_{0}$ and a non-vanishing holomorphic section $t$ of vector bundle $E_{T|_{\mathcal{M}}}$, such that $\langle t(w), e_{0}\rangle=1$ for any $w\in\Omega_0$.
For a fixed but arbitrary non-zero section $\widetilde{t}$ of $E_{T|_{\mathcal{M}}}$, there exist holomorphic functions $\{\phi_{i}\}_{n=0}^{\infty}$ on $\mathbb{D}$ such that $\widetilde{t}(\omega)=\sum\limits_{i=0}^{\infty}\phi_{i}(\omega)e_{i}$ for each $w\in\mathbb{D}$.
These functions do not have a common zero and each of them has at most finitely many of zeros.
Then we can find a connected open subset $\Omega_{0}$ of $\mathbb{D}$ such that $\phi_{0}(\omega)\neq 0$ for all $\omega\in \Omega_{0}$.
Let
$$t(\omega):=\frac{\widetilde{t}(\omega)}{\phi_{0}(\omega)}=e_0+\sum\limits_{i=1}^{\infty}\frac{\phi_{i}(\omega)}{\phi_{0}(\omega)}e_{i},\ w\in\Omega_0.$$
We also have $t(\omega)\in \ker(T|_{\mathcal{M}}-\omega)$ and $\langle t(\omega), e_{0}\rangle=1$ for all $\omega\in \Omega_{0}$.

Note that $(e_{0}\otimes e_{0})t(\omega)=\langle t(\omega), e_{0}\rangle e_{0}=e_{0}$ and
$$\begin{array}{lll}
\Big\langle\sum\limits_{i=0}^{n}(-1)^{i}{n \choose i}(T|_{\mathcal{M}})^{*i}(T|_{\mathcal{M}})^{i}t(w),t(w)\Big\rangle
&=&\sum\limits_{i=0}^{n}(-1)^{i}{n \choose i}\Big\langle w^{i}(T|_{\mathcal{M}})^{*i}t(w),t(w)\Big\rangle=(1-|\omega|^{2})^{n}\|t(w)\|^2
\end{array}$$
for $w\in\Omega_0$. From equation (\ref{dengs6}), we obtain that $\|t(w)\|^2=\frac{1}{(1-|\omega|^{2})^{n}}$, $w\in\Omega_0$. It follows that $\mathcal{K}_{T|_{\mathcal{M}}}(\omega)=-\frac{n}{(1-|\omega|^{2})^{2}}$.
Since $T\in \mathcal{L}(\mathcal{H})$ is $n-$hypercontractive, by using Theorem \ref{pro2}, we have $T$ is reducible.
\end{proof}

\end{document}